\documentclass[a4paper,10pt]{amsart}
\usepackage{amsmath,amssymb}
\usepackage{amscd}
\newtheorem{thm}{Theorem}[section]
\newtheorem{prop}[thm]{Proposition}
\newtheorem{lemma}[thm]{Lemma}
\newtheorem{cor}[thm]{Corollary}

\theoremstyle{remark}
\newtheorem{remark}[thm]{Remark}
\newcommand{\id}{{\rm{id}}}
\newcommand{\Ad}{{\rm{Ad}}}
\newcommand{\Hom}{{\rm{Hom}}}

\newcommand{\BC}{\mathbf C}

\newcommand{\BK}{\mathbf K}
\newcommand{\BB}{\mathbf B}
\newcommand{\la}{\langle}
\newcommand{\ra}{\rangle}

\newcommand{\Pic}{{\rm{Pic}}}
\newcommand{\Equi}{{\rm{Equi}}}
\newcommand{\Aut}{{\rm{Aut}}}
\newcommand{\Int}{{\rm{Int}}}
\newcommand{\Ima}{{\rm{Im}}}
\newcommand{\Ker}{{\rm{Ker}}}
\newcommand{\rU}{{\rm{U}}}
\newtheorem{Def}{Definition}[section]

\title{Equivariant Picard groups of $C^*$-algebras with finite dimensional $C^*$-Hopf algebra coactions}
\author{Kazunori Kodaka}
\address{Department of Mathematical Sciences, Faculty of Science, Ryukyu
\endgraf
University, Nishihara-cho, Okinawa, 903-0213, Japan}
\address{\sl{E-mail address}: \rm{kodaka@math.u-ryukyu.ac.jp}}

\begin{document}
\maketitle
\begin{abstract}
Let $A$ be a $C^*$-algebra and $H$ a finite dimensional $C^*$-Hopf algebra
with its dual $C^*$-Hopf algebra $H^0$. Let $(\rho, u)$ be a twisted coaction of $H^0$ on $A$.
We shall define the $(\rho, u, H)$-equivariant Picard group of $A$, which is denoted by
$\Pic_H^{\rho, u}(A)$, and discuss basic properties of $\Pic_H^{\rho, u}(A)$.
Also, we suppose that $(\rho, u)$ is the coaction of $H^0$ on the unital $C^*$-algebra $A$,
that is, $u=1\otimes 1^0$. We investigate the relation between $\Pic(A^s )$, the ordinary Picard group
of $A^s$ and $\Pic_H^{\rho^s}(A^s )$ where $A^s$ is the stable $C^*$-algebra of $A$
and $\rho^s$ is the coaction of $H^0$ on $A^s$ induced by $\rho$.
Furthermore, we shall show that $\Pic_{H^0}^{\widehat{\rho}}(A\rtimes_{\rho, u}H)$ is
isomorphic to $\Pic_H^{\rho, u}(A)$, where $\widehat{\rho}$ is the dual coaction of $H$ on
the twisted crossed product $A\rtimes_{\rho, u}H$ of $A$ by the twisted coaction $(\rho, u)$
of $H^0$ on $A$.

\end{abstract}

\section{Introduction}\label{sec:intro}Let $A$ be a $C^*$-algebra and $H$ a finite dimensional
$C^*$-Hopf algebra with its dual $C^*$-Hopf algebra $H^0$. Let $(\rho, u)$
be a twisted coaction of $H^0$ on $A$. We shall define the $(\rho, u, H)$-equivariant
Picard group of $A$, which is denoted by $\Pic_H^{\rho, u}(A)$.
Also, we shall give a similar result to the ordinary Picard group
as follows: Let $\Aut_H^{\rho, u}(A)$ be the group of all automorphisms $\alpha$ of $A$
satisfying that $(\alpha\otimes\id)\circ\rho=\rho\circ\alpha$ and let $\Int_H^{\rho, u}(A)$ be
the normal subgroup of $\Aut_H^{\rho, u}(A)$ consisting of all generalized inner
automorphisms $\Ad(v)$ of $A$ satisfying that $\rho(v)=v\otimes 1^0$, where
$v$ is a unitary element in the multiplier algebra $M(A)$ of $A$. Then we have the
following exact sequence:
$$
1\longrightarrow \Int_H^{\rho, u}(A)\longrightarrow \Aut_H^{\rho, u}(A)
\longrightarrow \Pic_H^{\rho, u}(A) .
$$
Especially, let $A^s$ be a stable $C^*$-algebra of a unital $C^*$-algebra $A$ and
$\rho$ a coaction of $H^0$ on $A$.
Also, let $\rho^s$ be the coaction of $H^0$ on $A^s$ induced by a coaction $\rho$
of $H^0$ on $A$. Then under a certain condition, we can obtain the exact
sequence
$$
1\longrightarrow \Int_H^{\rho^s}(A^s )\longrightarrow \Aut_H^{\rho^s}(A^s )
\longrightarrow \Pic_H^{\rho^s }(A^s )\longrightarrow 1 .
$$
In order to do this, we shall extend definitions and results in the case of unital $C^*$-algebras
to ones in the case of non-unital $C^*$-algebras in the section of Preliminaries.
Using this result, we shall investigate the relation between $\Pic (A^s )$, the ordinary Picard
group of $A^s$ and $\Pic_H^{\rho^s}(A^s )$, the $(\rho^s, H)$-equivariant Picard group of $A^s$.
Furthermore, we shall show that $\Pic_{H^0}^{\widehat{\rho}}(A\rtimes_{\rho, u}H)$
is isomorphic to $\Pic_H^{\rho, u}(A)$, where $\widehat{\rho}$ is the dual coaction of
$H$ on the twisted crossed product $A\rtimes_{\rho, u}H$ of $A$ by the twisted coaction
$(\rho, u)$.

\section{Preliminaries}\label{sec:pre}Let $H$ be a finite dimensional $C^*$-Hopf algebra.
We denote its comultiplication, counit and antipode by $\Delta$, $\epsilon$ and $S$, respectively. We shall
use Sweedler's notation $\Delta(h)=h_{(1)}\otimes h_{(2)}$ for any $h\in H$
which suppresses a possible summation when we write comultiplications. We denote by $N$
the dimension of $H$. Let $H^0$ be the dual $C^*$-Hopf algebra of $H$. We denote its
comultiplication, counit and antipode by $\Delta^0$, $\epsilon^0$ and $S^0$, respectively. There is the distinguished
projection $e$ in $H$. We note that $e$ is the Haar trace on $H^0$. Also, there is the distinguished projection
$\tau$ in $H^0$ which is the Haar trace on $H$. Since $H$ is finite dimensional,
$H\cong\oplus_{k=1}^L M_{f_k }(\BC)$ and $H^0 \cong\oplus_{k=1}^K M_{d_k }(\BC)$ as $C^*$-algebras.
Let $\{\, v_{ij}^k \, | \, k=1,2,\dots,L, \, i,j=1,2,\dots,f_k \, \}$
be a system of matrix units of $H$. Let $\{\, w_{ij}^k \, |\, k=1,2,\dots,K, \, i,j=1,2,\dots,d_k \}$
be a basis of $H$ satisfying Szyma\'nski and Peligrad's \cite [Theorem 2.2,2]{SP:saturated},
which is called  a system of {\it comatrix} units of $H$, that is,
the dual basis of a system of matrix units of $H^0$. Also let $\{\, \phi_{ij}^k \, |\, k=1,2,\dots,K, \, i,j=1,2,\dots,d_k \}$
and $\{\, \omega_{ij}^k \, |\, k=1,2,\dots,L, \, i,j=1,2,\dots,f_k \}$ be systems of matrix units and comatrix units of $H^0$,
respectively.
\par
Let $A$ be a $C^*$-algebra and $M(A)$ its multiplier algebra. Let $p, q$ be
projections in $A$. If $p$ and $q$ are Murray-von Neumann equivalent,
then we denote it by $p\sim q$ in $A$. We denote by $\id_A$ and $1_A$ the identity map on $A$ and the unit element
in $A$, respectively. We simply denote them by $\id$ and $1$ if
no confusion arises. Modifying Blattner, Cohen and
Montgomery \cite [Definition 2.1]{BCM:crossed}, we shall define a weak coaction of $H^0$ on $A$.

\begin{Def}\label{Def:wcoaction}By a weak coaction of $H^0$ on $A$, we mean a $*$-homomorphism
$\rho:A\to A\otimes H^0$ satisfying the following conditions:
\newline
(1) $\overline{\rho(A)(A\otimes H^0 )}=A\otimes H^0$,
\newline
(2) $(\id\otimes\epsilon^0 )(\rho(x))=x$ for any $x\in A$,
\par
By a coaction of $H^0$ on $A$, we mean a weak coaction $\rho$ such that
\newline
(3) $(\rho\otimes\id)\circ\rho =(\id\otimes \Delta^0 )\circ\rho$.
\end{Def}

By Condition (1) in Definition \ref{Def:wcoaction}, for any approximate unit $\{u_{\alpha}\}$ of
$A$ and $x\in A\otimes H^0$, $\rho(u_{\alpha})x\to x$ $(\alpha\to\infty)$. Hence $\rho(1)=1\otimes 1^0$ when
$A$ is unital. Since $H^0$
is finite dimensional, $M(A\otimes H^0 )\cong M(A)\otimes H^0$.
We identify $M(A\otimes H^0 )$ with $M(A)\otimes H^0$. We also identify $M(A\otimes H^0 \otimes H^0 )$
with $M(A)\otimes H^0 \otimes H^0$. Let $\rho$ be a weak coaction of
$H^0$ on $A$. By Jensen and Thomsen \cite [Corollary 1.1.15]{JT:KK}, there is the unique strictly continuous
homomorphism $\underline{\rho}: M(A)\to M(A)\otimes H^0$ extending $\rho$.

\begin{lemma}\label{lem:extension}With the above notations, $\underline{\rho}$ is a weak
coaction of $H^0$ on $M(A)$.
\end{lemma}
\begin{proof}
Clearly $\underline{\rho}$ is a $*$-homomorphism of $M(A)$ to $M(A)\otimes H^0$.
Let $\{u_{\alpha}\}$ be an approximate unit of $A$. Then by Condition (1) in Definition \ref{Def:wcoaction},
$\{\rho(u_{\alpha})\}$ is an approximate unit of $A\otimes H^0$.
Hence $\underline{\rho}(1)=1\otimes 1^0$. Since $H^0$ is finite dimensional,
$\id\otimes\epsilon^0$ is strictly continuous. Hence $\underline{\rho}$ satisfies
Condition (2) in Definition \ref{Def:wcoaction}.
\end{proof}

Let $\rho$ be a weak coaction of $H^0$ on $A$ and $u$ a unitary element
in $M(A)\otimes H^0 \otimes H^0$. Following Masuda and Tomatsu
\cite [Section 3]{MT:Kac}, we shall define a twisted coaction of $H^0$ on $A$.

\begin{Def}\label{Def:tcoaction}The pair $(\rho, u)$ is a 
\sl
twisted coaction
\rm
of $H^0$ on $A$ if the following conditions hold:
\newline
(1) $(\rho\otimes \id)\circ\rho=\Ad(u)\circ (\id\otimes\Delta^0 )\circ\rho$,
\newline
(2) $(u\otimes 1^0 )(\id\otimes\Delta^0 \otimes \id)(u)
=(\underline{\rho}\otimes \id\otimes \id)(u)(\id\otimes \id\otimes\Delta^0 )(u)$,
\newline
(3) $(\id\otimes \id \otimes\epsilon^0 )(u)=(\id\otimes\epsilon^0 \otimes \id)(u)
=1\otimes 1^0$.
\end{Def}

\begin{remark}\label{remark:extension}Let $(\rho, u)$ be a twisted coaction of $H^0$ on $A$.
Since $H^0$ is finite dimensional, $\id_{M(A)}\otimes \Delta^0$
is strictly continuous. Hence by Lemma \ref{lem:extension}, $(\underline{\rho}, u)$ satisfies
Definition \ref{Def:tcoaction}. Therefore, $(\underline{\rho}, u)$ is a twisted coaction of
$H^0$ on $M(A)$. Hence if $\rho$ is a coaction of $H^0$ on $A$, $\underline{\rho}$ is
a coaction of $H^0$ on $M(A)$.
\end{remark}

Let $\Hom (H, M(A))$ be the linear space of all linear maps from $H$ to $M(A)$.
Then by Sweedler \cite [pp69-70]{Sweedler:Hopf}, it becomes a unital
convolution *-algebra. Similarly, we define $\Hom (H\times H, M(A))$. We note
that $\epsilon$ and $\epsilon\otimes\epsilon$ are the unit elements in $\Hom (H, M(A))$
and $\Hom (H\times H, M(A))$, respectively.

Modifying Blattner, Cohen and Montgomery, \cite [Definition 1.1]{BCM:crossed}
we shall define a weak action of $H$ on $A$.

\begin{Def}\label{Def:waction}By a weak action of $H$ on $A$ , we mean a bilinear
map $(h, x)\mapsto h\cdot x$ of $H\times A$ to $A$ satisfying the following conditions:
\newline
(1) $h\cdot (xy)=[h_{(1)}\cdot x][h_{(2)}\cdot y]$ for any $h\in H$, $x, y\in A$,
\newline
(2) $[h\cdot u_{\alpha}]x\to\epsilon (h)x$ for any approximate unit $\{u_{\alpha}\}$
of $A$ and $x\in A$.
\newline
(3) $1\cdot x=x$ for any $x\in A$,
\newline
(4) $[h\cdot x]^* =S(h)^* \cdot x^*$ for any $h\in H$, $x\in A$.
\par
By an action of $H$ on $A$, we mean a weak action of such that
\newline
(5) $h\cdot [l\cdot x]=(hl)\cdot x$ for any $x\in A$ and $h, l\in H$.
\end{Def}

Since $H$ is finite dimensional, as mentioned in
\cite [pp. 163]{BCM:crossed}, there is an isomorphism $\imath$ of
$M(A)\otimes H^0$ onto $\Hom (H, M(A))$ defined by $\imath(x\otimes \phi)(h)=\phi(h)x$
for any $x\in M(A)$, $h\in H$, $\phi\in H^0$. Also, we can define
an isomorphism $\jmath$ of $M(A)\otimes H^0 \otimes H^0$ onto $\Hom (H\times H, M(A))$
in the similar way to the above. We note that $\imath(A\otimes H^0 )=\Hom(H, A)$ and
$\jmath (A\otimes H^0 \otimes H^0 )=\Hom (H\otimes H, A)$. For any $x\in M(A)\otimes H^0$
and $y\in M(A)\otimes H^0 \otimes H^0 $, we denote $\imath(x)$ and $\jmath(y)$
by $\widehat{x}$ and $\widehat{y}$, respectively.
\par
Let a bilinear map $(h, x)\mapsto h\cdot x$ from $H\times A$ to $A$ be
a weak action. For any $x\in A$, let $f_x$ be the linear map from $H$ to $A$
defined by $f_x (h)=h\cdot x$ for any $h\in H$. Let $\rho$ be the linear map
from $A$ to $A\otimes H^0$ defined by $\rho(x)=\imath^{-1}(f_x )$ for any
$x\in A$.

\begin{lemma}\label{lem:tocoaction}With the above notations, $\rho$ is a weak coaction of
$H^0$ on $A$.
\end{lemma}
\begin{proof}By its definition, $\rho$ is a $*$-homomorphism of $A$ to $A\otimes H^0$
satisfying Conditions (2) in Definiotn \ref{Def:wcoaction}.
So, we have only to show that $\rho$ satisfies Condition (1) in Definition \ref{Def:wcoaction}. Let
$\{u_{\alpha}\}$ be an approximate unit of $A$. We can write that
$\rho(u_{\alpha})=\sum_j u_{\alpha j}\otimes\phi_j$, where $u_{\alpha j}\in A$ and
$\{\phi_j \}$ is a basis of $H^0$ with $\sum_j \phi_j =1^0 $. Let $\{h_j \}$ be the
dual basis of $H$ corresponding to $\{\phi_j \}$. Then for any $x\in A$ and $j$,
$[h_j \cdot u_{\alpha}]x\to \epsilon (h_j )x$ by Condition (2) in Definition \ref{Def:waction}.
Since $[h_j \cdot u_{\alpha}]x=(\id\otimes h_j )(\rho(u_{\alpha}))x=u_{\alpha j}x$,
$u_{\alpha j}x\to\epsilon (h_j )x$ for any $j$. Also, since $\sum_j \phi_j =1^0 $,
$$
1=\phi_j (h_j )=\sum_i \phi_i (h_j )=1^0 (h_j )=\epsilon (h_j )
$$
for any $j$. Hence $u_{\alpha j}x\to x$ for any $j$. Therefore, for any $x\in A$ and $\phi\in H^0 $,
$$
\rho(u_{\alpha})(x\otimes\phi)=\sum_j u_{\alpha j}x\otimes\phi_j \phi \to \sum_j x\otimes\phi_j \phi
=x\otimes \phi .
$$
Thus $\overline{\rho(A)(A\otimes H^0 )}=A\otimes H^0 $.
\end{proof}

For any weak coaction $\rho$ of $H^0$ on $A$, we can define the bilinear
map $(h, x)\mapsto h\cdot_{\rho}x$ from $H\times A$ to $A$ by
$$
h\cdot_{\rho}x=(\id\otimes h)(\rho(x))=\rho(x)^{\widehat{}}(h) .
$$
We shall prove that the above map is a weak action of $H$ on $A$.

\begin{lemma}\label{lem:toaction}With the above notations, the linear map
$(h, x)\mapsto h\cdot_{\rho}x$ from $H\times A$ to $A$ is a weak action of $H$
on $A$.
\end{lemma}
\begin{proof}We have only to show that the above linear map satisfies Condition (2) in
Definition \ref{Def:waction}. Let $\{u_{\alpha}\}$ be an approximate unit of $A$. Then for any
$x\in A\otimes H^0$, $\rho(u_{\alpha})x\to x$ by the proof of Lemma \ref{lem:extension}. We can write that $\rho(u_{\alpha})
=\sum_j u_{\alpha j }\otimes \phi_j $, where $u_{\alpha j }\in A$ and $\{\phi_j \}$
is a basis of $H^0$. Then for any $a\in A$,
\begin{align*}
[h\cdot_{\rho}u_{\alpha}]a & =(\id\otimes h)(\rho(u_{\alpha}))a=\sum_j u_{\alpha j}\phi_j (h)a \\
& =(\id\otimes h)(\rho(u_{\alpha})(a\otimes 1^0 ))\to\epsilon(h)a
\end{align*}
since $\id\otimes h$ is a bounded operator from $A\otimes H^0$ to $A$.
\end{proof}
\begin{remark}\label{remark:some}By the proofs of Lemmas \ref{lem:tocoaction} and \ref{lem:toaction},
Condition (2) in Definition \ref{Def:waction} is equivalent to
the following Condition (2)':
\newline
(2)' $[h\cdot u_{\alpha}]x\to\epsilon(h)x$ for some approximate unit of $A$ and any $x\in A$.
\newline
Also, if $A$ is unital, Condition (2) in Definition \ref{Def:waction} means that $h\cdot 1=\epsilon(h)$
for any $h\in H$.
\end{remark}

Let $\rho$ be a weak coaction of $H^0$ on $A$. Then by Lemma \ref{lem:toaction}, there
is a weak action of $H$ on $A$. We call it the weak action of $H$ on $A$
\sl
induced by $\rho$.
\rm
Also, by Lemma \ref{lem:extension}, there is the weak coaction $\underline{\rho}$
of $H^0$ on $M(A)$, which is an extension of $\rho$ to $M(A)$. Hence we can obtain
the action of $H$ on $M(A)$ induced by $\underline{\rho}$. We can see that this action
is an extension of the action induced by $\rho$ to $M(A)$.

\begin{Def}\label{cocycle}Let $\sigma: H\times H\to M(A)$ be a bilinear map.
$\sigma$ is a
\it
unitary cocycle
\rm
for a weak action of $H$ on $A$ if $\sigma$ satisfies the
following conditions:
\newline
(1) $\sigma$ is a unitary element in $\Hom (H\times H, M(A))$,
\newline
(2) $\sigma$ is normal, that is, for any $h\in H$, $\sigma(h, 1)=\sigma(1, h)=\epsilon(h)1$,
\newline
(3) (cocycle condition) For any $h, l, m\in H$,
\par
$[h_{(1)}\cdot\sigma(l_{(1)}, m_{(1)})]\sigma(h_{(2)}, l_{(2)}m_{(2)})=\sigma(h_{(1)}, l_{(1)})\sigma(h_{(2)}l_{(2)}, m)$,
\newline
(4) (twisted modular condition) For any $h, l\in H$, $x\in A$,
\par
$[h_{(1)}\cdot[l_{(1)}\cdot x]]\sigma(h_{(2)}, l_{(2)})=\sigma(h_{(1)}, l_{(1)})[(h_{(2)}l_{(2)})\cdot x]$,
\newline
where if necessary, we consider the extension of the weak action to $M(A)$.
\end{Def}

We call a pair which is consisting of a weak action of $H$ on $A$ and its unitary
cocycle, a
\it
twisted action
\rm
of $H$ on $A$.
\par
Let $(\rho, u)$ be a twisted coaction of $H^0$ on $A$. Then we can consider the twisted
action of $H$ on $A$ and its unitary cocycle $\widehat{u}$ defined by
$$
h\cdot_{\rho, u}x =\rho(x)^{\widehat{}}(h)=(\id\otimes h)(\rho(x))
$$
for any $x\in A$ and $h\in H$. We call it the twisted
action
\sl
induced by $(\rho, u)$.
\rm
 Also, we can consider the twisted coaction
$(\underline{\rho}, u)$ of $H^0$ on $M(A)$ and the twisted action of $H$
on $M(A)$ induced by $(\underline{\rho}, u)$. Let $M(A)\rtimes_{\underline{\rho}, u}H$
be the twisted crossed product by the twisted action of $H$ on $M(A)$ induced
by $(\underline{\rho}, u)$. Let $x\rtimes_{\underline{\rho}, u}h$ be the element in
$M(A)\rtimes_{\underline{\rho}, u}H$ induced by  elements $x\in M(A)$, $h\in H$.
Let $A\rtimes_{\rho, u}H$ be the set of all finite sums of elements
in the form $x\rtimes_{\underline{\rho}, u}h$, where $x\in A$, $h\in H$. By easy computations,
we can see that $A\rtimes_{\rho, u}H$ is a closed two-sided ideal of
$M(A)\rtimes_{\underline{\rho}, u}H$. We call it the twisted crossed product by $(\rho, u)$ and its 
element denotes by $x\rtimes_{\rho, u}h$, where $x\in A$ and $h\in H$.
Let $E_1^{\underline{\rho}, u}$ be the canonical conditional expectation from
$M(A)\rtimes_{\underline{\rho}, u}H$ onto $M(A)$ defined by
$$
E_1^{\underline{\rho}, u}(x\rtimes_{\underline{\rho}, u}h)=\tau(h)x
$$
for any $x\in M(A)$ and $h\in H$. Let $\Lambda$ be the set of all triplets $(i, j, k)$,
where $i, j=1,2,\dots, d_k$ and $k=1,2,\dots, K$ with $\sum_{k=1}^K d_k^2 =N$.
Let $W_I =\sqrt{d_k}\rtimes_{\underline{\rho}, u}w_{ij}^k$ for any $I=(i, j, k)\in\Lambda$.
By \cite [Proposition 3.18]{KT1:inclusion}, $\{(W_I^* , \, W_I )\}_{I\in\Lambda}$
is a quasi-basis for $E_1^{\underline{\rho}, u}$. We suppose that $A$ acts on a Hilbert space
faithfully and nondegenerately.

\begin{lemma}\label{lem:multiplier}With the above notations, $M(A)\rtimes_{\underline{\rho}, u}H
=M(A\rtimes_{\rho, u}H)$.
\end{lemma}
\begin{proof}By the definition of the multiplier algebras $M(A)$ and  $M(A\rtimes_{\rho, u}H)$,
it is clear that $M(A)\rtimes_{\underline{\rho}, u}H\subset M(A\rtimes_{\rho, u}H)$
since $M(A)\rtimes_{\underline{\rho}, u}H$ and $M(A\rtimes_{\rho, u}H)$ act on the same Hilbert space.
We show that another inclusion. Let $x\in M(A\rtimes_{\rho, u}H)$.
Then there is a bounded net $\{x_{\alpha}\}_{\alpha\in\Gamma}\subset A\rtimes_{\rho, u}H$ such that
$\{x_{\alpha}\}_{\alpha\in\Gamma}$ converges to $x$ strictly. Since $x_{\alpha}\in A\rtimes_{\rho, u}H$,
$x_{\alpha}=\sum_I E_1^{\underline{\rho}, u}(x_{\alpha}W_I^* )W_I$. By the definition of $E_1^{\underline{\rho}, u}$,
$E_1^{\underline{\rho}, u}(x_{\alpha}W_I^* )\in A$. Also, for any $a\in A$,
$$
\lim_{\alpha\to\infty}E_1^{\underline{\rho}, u}(x_{\alpha}W_I^* )a =\lim_{\alpha\to\infty}E_1^{\underline{\rho}, u}
(x_{\alpha}W_I^* a)=E_1^{\underline{\rho}, u}(xW_I^* a) \\
=E_1^{\underline{\rho}, u}(xW_I^* )a .
$$
Similarly, $\lim_{\alpha\to\infty}aE_1^{\underline{\rho}, u}(x_{\alpha}W_I^* )
=aE_1^{\underline{\rho}, u}(xW_I^* )$.
Hence $E_1^{\underline{\rho}, u}(xW_I^* )\in M(A)$.
Also, by the above discussions, we can see that $E_1^{\underline{\rho}, u}(\, \cdot \, \, W_I^*  )$ is strictly continuous
for any $I\in \Lambda$.
For any $a\in A$ and $h\in H$,
\begin{align*}
(a\rtimes_{\rho, u}h)E_1^{\underline{\rho}, u}(x_{\alpha}W_I^* ) & =a[h_{(1)}\cdot_{\rho, u}E_1^{\underline{\rho},u}
(x_{\alpha}W_I^* )]\rtimes_{\rho, u}h_{(2)} \\
& =a((\id\otimes h_{(1)})\circ(\underline{\rho}\otimes\id))(E_1^{\underline{\rho}, u}(x_{\alpha}W_I^* )
\rtimes_{\rho, u}h_{(2)}) .
\end{align*}
Since $\id\otimes h_{(1)}$, $\underline{\rho}\otimes\id$ and $E_1^{\underline{\rho}, u}(\, \cdot  \,\, W_I )$ are
strictly continuous for any $I\in\Lambda$, we can see that
$$
\lim_{\alpha\to\infty}(a\rtimes_{\rho, u}h)E_1^{\underline{\rho}, u}(x_{\alpha}W_I^* )
=(a\rtimes_{\rho, u}h)E_1^{\underline{\rho}, u}(xW_I^* ) .
$$
Similarly, we can see that for any $a\in A$, $h\in H$,
$$
\lim_{\alpha\to\infty}E_1^{\underline{\rho}, u}(x_{\alpha}W_I^* )(a\rtimes_{\rho, u}h)
=E_1^{\underline{\rho}, u}(xW_I^* )(a\rtimes_{\rho, u}h).
$$
Thus $E_1^{\underline{\rho}, u}(x_{\alpha}W_I^* )$ converges to $E_1^{\underline{\rho}, u}(xW_I^* )$
in $M(A\rtimes_{\rho, u}H)$ strictly. Therefore, $x=\sum_I E_1^{\underline{\rho}, u}(xW_I^* )W_I$
since $x_{\alpha}=\sum_I E_1^{\underline{\rho}, u}(x_{\alpha}W_I^* )W_I$.
It follows that $x\in M(A)\rtimes_{\underline{\rho}, u}H$.
\end{proof}

\begin{remark}\label{remark:dual}Let $(\underline{\rho})^{\widehat{}}$ be the dual coaction of $\underline{\rho}$ of
$H$ on $M(A)\rtimes_{\underline{\rho}, u}H$ and $\underline{(\widehat{\rho})}$ the coaction of $H$
on $M(A\rtimes_{\rho, u}H)$ induced by the dual coaction $\widehat{\rho}$ of $H$ on $A\rtimes_{\rho, u}H$.
By Lemma \ref{lem:multiplier}, we can see that $(\underline{\rho})^{\widehat{}}=\underline{(\widehat{\rho})}$.
Indeed, by Lemma \ref{lem:multiplier}, it suffices to show that
$\underline{(\widehat{\rho})}(x\rtimes_{\underline{\rho}, u}h)
=(\underline{\rho})^{\widehat{}}(x\rtimes_{\underline{\rho}, u}h)$ for any
$x\in M(A)$ and $h\in H$. Since $x\in M(A)$, there is a bounded net $\{x_{\alpha}\}\subset A$
such that $x_{\alpha}$ converges to $x$ strictly in $M(A)$. Then since
$x_{\alpha}\rtimes_{\rho, u}h$ converges to $x\rtimes_{\underline{\rho}, u}h$ strictly in $M(A)\rtimes_{\underline{\rho}, u}H$ and
$\underline{(\widehat{\rho})}$ is strictly continuous,
\begin{align*}
\underline{(\widehat{\rho})}(x\rtimes_{\underline{\rho}, u}h) & =
\lim_{\alpha\to\infty}\widehat{\rho}(x_{\alpha}\rtimes_{\rho, u}h)
=\lim_{\alpha\to\infty}(x_{\alpha}\rtimes_{\rho, u}h_{(1)})\otimes h_{(2)} \\
& =(x\rtimes_{\underline{\rho}, u}h_{(1)})\otimes h_{(2)}
=(\underline{\rho})^{\widehat{}}(x\rtimes_{\underline{\rho}, u}h) ,
\end{align*}
where the limits are taken under the strict topology. We denote it by
$\underline{\widehat{\rho}}$.
\end{remark}

Next, we extend \cite [Theorem 3.3]{KT2:coaction} to a twisted coaction
of $H^0$ on a (non-unital) $C^*$-algebra $A$. Before doing it, we
define the exterior equivalence for twisted coactions of a finite dimensional
$C^*$-Hopf algebra $H^0$ on a $C^*$-algebra $A$.

\begin{Def}\label{Def:exterior}Let $(\rho, u)$ and $(\sigma, v)$ be twisted coactions
of $H^0$ on $A$. We say that $(\rho, u)$ is 
\sl
exterior equivalent
\rm
to $(\sigma, v)$ if there is a unitary element $w\in M(A)\otimes H^0$
satisfying the following conditions:
\newline
(1) $\sigma=\Ad(w)\circ\rho$,
\newline
(2) $v=(w\otimes 1^0 )(\underline{\rho}\otimes\id)(w)u(\id\otimes\Delta^0 )(w^* )$.
\end{Def}

The above conditions (1), (2) are equivalent to the following, respectively:
\newline
(1)' $h\cdot_{\sigma, v}a =\widehat{w}(h_{(1)})[h_{(2)}\cdot_{\rho, u}a]\widehat{w}^* (h_{(3)})$
for any $a\in A$ and $h\in H$,
\newline
(2)' $\widehat{v}(h, l)=\widehat{w}(h_{(1)})[h_{(2)}\cdot_{\underline{\rho}, u}\widehat{w}(l_{(1)})]
\widehat{u}(h_{(3)}, l_{(2)})\widehat{w}^* (h_{(4)}l_{(3)})$ for any $h, l\in H^0$.
\newline
If $\rho$ and $\sigma$ are coactions of $H^0$ on $A$, Conditions (1), (2) and (1)', (2)' are following:
\newline
(1) $\sigma=\Ad(w)\circ\rho$,
\newline
(2) $(w\otimes 1^0 )(\underline{\rho}\otimes\id)(w)=(\id\otimes\Delta^0 )(w)$,
\newline
(1)' $h\cdot_{\sigma}a =\widehat{w}(h_{(1)})[h_{(2)}\cdot_{\rho}a]\widehat{w}^* (h_{(3)})$
for any $a\in A$, $h\in H^0$,
\newline
(2)' $\widehat{w}(h_{(1)})[h_{(2)}\cdot_{\underline{\rho}}\widehat{w}(l)]=\widehat{w}(hl)$
for any $h, l\in H^0$.
\newline
Furthermore, let $(\rho, u)$ be a twisted coaction of $H^0$ on $A$ and let $w$ be any
unitary element in $M(A)\otimes H^0$ with $(\id\otimes\epsilon^0 )(w)=1^0 $.
Let
$$
\sigma=\Ad(w)\circ\rho , \quad v=(w\otimes 1^0 )(\underline{\rho}\otimes\id)(w)u(\id\otimes\Delta^0 )(w^* ) .
$$
Then $(\sigma, v)$ is a twisted coaction of $H^0$ on $A$ by easy computations.

In the case of twisted coactions on von Neumann algebras, Vaes and Vanierman \cite {VV:bicrossed}
and in the coase of ordinary coactions on $C^*$-algebras,
Baaj and Skandalis \cite {BS:Kasparov} have already obtained the much more generalized results than the following.
We give a proof related to inclusions of unital $C^*$-algebras of Watatani index-finite type.

\begin{prop}\label{prop:nonunital}Let $A$ be a $C^*$-algebra and
$H$ a finite dimensional $C^*$-Hopf algebra with its dual $C^*$-Hopf
algebra $H^0$. Let $(\rho, u)$ be a twisted coaction of $H^0$ on $A$.
Then there are an isomorphism $\Psi$ of $M(A)\otimes M_N (\BC)$
onto $M(A)\rtimes_{\underline{\rho}, u}H\rtimes_{\underline{\widehat{\rho}}}H^0$ and
a unitary element $U\in (M(A)\rtimes_{\underline{\rho}, u}H\rtimes_{\underline{\widehat{\rho}}}H^0 )
\otimes H^0$ such that
\begin{align*}
& \Ad(U)\circ\widehat{\widehat{\rho}}=(\Psi\otimes\id_{H^0})\circ(\rho\otimes\id_{M_N (\BC)})\circ\Psi^{-1} , \\
& (\Psi\otimes\id_{H^0}\otimes\id_{H^0})(u\otimes I_N )=(U\otimes 1^0 )(\underline{\widehat{\widehat{\rho}}}
\otimes\id_{H^0})(U)(\id\otimes\Delta^0 )(U^* ) , \\
& \Psi(A\otimes M_N (\BC))=A\rtimes_{\rho, u}H\rtimes_{\widehat{\rho}}H^0 .
\end{align*}
That is, the coaction $\widehat{\widehat{\rho}}$ of $H^0$ on $A\rtimes_{\rho, u}\rtimes_{\widehat{\rho}}H^0$
is exterior equivalent to the twisted coaction
$$
((\Psi\otimes\id_{H^0})\circ(\rho\otimes\id_{M_N (\BC)})\circ\Psi^{-1}, \,
(\Psi\otimes\id_{H^0}\otimes\id_{H^0})(u\otimes I_N )) ,
$$
where we identify $A\otimes H^0 \otimes H^0 \otimes M_N (\BC)$ with $A\otimes M_N (\BC)\otimes H^0 \otimes H^0$.
\end{prop}
\begin{proof}
By \cite [Theorem 3.3]{KT2:coaction}, there are an isomorphism $\Psi$ of $M(A)\otimes M_N (\BC)$ onto
$M(A)\rtimes_{\underline{\rho}, u}H\rtimes_{\underline{\widehat{\rho}}}H^0$ and a unitary
element $U\in (M(A)\rtimes_{\underline{\rho}, u}H\rtimes_{\underline{\widehat{\rho}}}H^0 )\otimes H^0$
satisfying the required conditions except for the equation
$$
\Psi(A\otimes M_N (\BC))=A\rtimes_{\rho, u}H\rtimes_{\widehat{\rho}}H^0 .
$$
So, we show the equation.
By \cite [Section 3]{KT2:coaction},
$$
\Psi([a_{IJ}])=\sum_{I, J}V_I^* (a_{IJ}\rtimes_{\rho, u}1\rtimes_{\widehat{\rho}}1^0 )V_J
$$
for any $[a_{IJ}]\in A\otimes M_N (\BC)$, where $V_I
=(1\rtimes_{\underline{\widehat{\rho}}}\tau)(W_I \rtimes_{\underline{\widehat{\rho}}}1^0 )$
for any $I\in\Lambda$. Since $V_I\in M(A)\rtimes_{\underline{\rho}, u}H\rtimes_{\underline{\widehat{\rho}}}H^0$
for any $I\in\Lambda$, $\Psi(A\otimes M_N (\BC))\subset A\rtimes_{\rho, u}\rtimes_{\widehat{\rho}}H^0$.
For any $z\in A\rtimes_{\rho, u}H\rtimes_{\widehat{\rho}}H^0$, we can write that
$$
z=\sum_{i=1}^n (x_i \rtimes_{\widehat{\rho}}1^0 )(1\rtimes_{\widehat{\rho}}\tau)(y_i \rtimes_{\widehat{\rho}}1^0 ) ,
$$
where $x_i , y_i \in M(A)\rtimes_{\underline{\rho}, u}H$ for any $i$.
Let $\{u_{\alpha}\}$ be an approximate unit of $A$. Then
$(u_{\alpha}\rtimes_{\rho, u}1\rtimes_{\widehat{\rho}}1^0 )(x_i \rtimes_{\widehat{\rho}}1^0 )$ and
$(y_i \rtimes_{\widehat{\rho}}1^0 )(u_{\alpha}\rtimes_{\rho, u}1\rtimes_{\widehat{\rho}}1^0 )$ are
in $A\rtimes_{\rho, u}H\rtimes_{\widehat{\rho}}H^0$ for any $i$ and $\alpha$.
Hence $(A\rtimes_{\rho, u}H\rtimes_{\widehat{\rho}}1^0 )(1\rtimes_{\widehat{\rho}}\tau)
(A\rtimes_{\rho, u}H\rtimes_{\widehat{\rho}}1^0 )$ is dense in $A\rtimes_{\rho, u}H\rtimes_{\widehat{\rho}}H^0$.
On the other hand, for any $x, y \in A\rtimes_{\rho, u }H$,
$$
\Psi([E_1^{\underline{\rho}, u}(W_I x)E_1^{\underline{\rho}, u}(yW_J^* )]_{I, J})
=(x\rtimes_{\widehat{\rho}}1^0 )(1\rtimes_{\widehat{\rho}}\tau)(y\rtimes_{\widehat{\rho}}1^0 )
$$
by the proof of \cite[Theorem 3.3]{KT2:coaction}. Since $E_1^{\underline{\rho}, u}(A\rtimes_{\rho, u}H)=A$
and $E_1^{\underline{\rho}, u}$ is continuous by its definition,
$A\rtimes_{\rho, u}H\rtimes_{\widehat{\rho}}H^0 \subset\Psi(A\otimes M_N (\BC))$.
\end{proof}

We extend \cite [Theorem 6.4]{KT1:inclusion} to coactions of $H^0$ on a (non-unital) $C^*$-algebra.
First, we recall a saturated coaction. We say that a coaction $\rho$ of $H^0$ on a unital $C^*$-algebra $A$ is
\sl
saturated
\rm
if the induced action from $\rho$ of $H$ on $A$ is saturated in the sense of \cite [Definition 4.2]{SP:saturated}.
\par
Let $B$ be a $C^*$-algebra and $\sigma$ a coaction of $H^0$ on $B$.
Let $B^{\sigma}=\{b\in B \, | \, \sigma(b)=b\otimes 1^0 \}$, the fixed point $C^*$-subalgebra of $B$
for the coaction $\sigma$. We suppose that $B$ acts on a Hilbert space $\mathcal{H}$
non-degenerately and faithfully. Also, we suppose that $\underline{\sigma}$ is saturated.
Then the canonical conditional expectation $E^{\underline{\sigma}}$ from $M(B)$
onto $M(B)^{\underline{\sigma}}$ defined by $E^{\underline{\sigma}}(x)=e\cdot_{\underline{\sigma}}x$
for any $x\in M(B)$ is of Watatani index-finite type by \cite [Theorem 4.3]{SP:saturated}.
Thus there is a quasi-basis $\{(u_i, u_i^* )\}_{i=1}^n$ of $E^{\underline{\sigma}}$.
Let $\{v_{\alpha}\}$ be an approximate unit of $B^{\sigma}$. For any $x\in B$,
$$
v_{\alpha}x =v_{\alpha}\sum_{i=1}^n E^{\underline{\sigma}}(xu_i )u_i^* 
\to\sum_{i=1}^n E^{\underline{\sigma}}(xu_i )u_i^* =x \quad (\alpha\to\infty)
$$
since $E^{\sigma}(xu_i )\in B^{\sigma}$.
Similarly $xv_{\alpha}\to x$ $(\alpha\to\infty)$ since $x=\sum_{i=1}^n u_i E^{\underline{\sigma}}(u_i^* x)$.
Thus $\{v_{\alpha}\}$ is an approximate
unit of $B$. Hence $B^{\sigma}$ acts on $\mathcal{H}$ non-degenerately and faithfully.
 
\begin{lemma}\label{lem:fix}With the above notations, we suppose that
$\underline{\sigma}$ is saturated. Then $M(B^{\sigma})=M(B)^{\underline{\sigma}}$.
\end{lemma}
\begin{proof}By the above discussions, we may suppose that $B$ and $B^{\sigma}$ act on a Hilbert space
non-degenerately and faithfully. Let $x\in M(B^{\sigma})$.
Then there is a bounded net $\{a_{\alpha}\}\subset B^{\sigma}$ such that
$a_{\alpha}\to x$ $(\alpha\to\infty)$ strictly in $M(B^{\sigma})$. Since any approximate unit of $B^{\sigma}$
is an approximate unit of $B$ by the above discussion, for any $y\in B^{\sigma}$,
$$
\underline{\sigma}(x)(y\otimes 1^0 )=\sigma(xy)=\sigma(\lim_{\alpha\to\infty}a_{\alpha}y)=\lim_{\alpha\to\infty}\sigma(a_{\alpha}y)
=\lim_{\alpha\to\infty}a_{\alpha}y\otimes 1^0 =xy\otimes 1^0 .
$$
Thus $x\in M(B)^{\underline{\sigma}}$.
Next, let $x\in M(B)^{\underline{\sigma}}$. Then for any $b\in B^{\sigma}$, $xb$ and $bx$ are in $B$.
Thus
\begin{align*}
\sigma(xb) & =\underline{\sigma}(x)\sigma(b)=(x\otimes 1^0 )(b\otimes 1^0 )=xb\otimes 1^0 , \\
\sigma(bx) & =\sigma(b)\underline{\sigma}(x)=(b\otimes 1^0 )(x\otimes 1^0 )=bx\otimes 1^0 .
\end{align*}
Hence $x\in M(B^{\sigma})$.
\end{proof}

We suppose that $\widehat{\underline{\sigma}}(1\rtimes_{\underline{\sigma}}e)
\sim(1\rtimes_{\underline{\sigma}}e)\otimes 1$ in $(M(B)\rtimes_{\underline{\sigma}}H)\otimes H$.
As mentioned in \cite [Section 2]{KT2:coaction}, without the assumption of saturation for
an action, all the statements in Sections 4, 5 and 6 in \cite {KT1:inclusion} hold.
Hence by \cite [Sections 4 and 5]{KT1:inclusion}, $\underline{\sigma}$ is
saturated and 
there is a unitary element $w^{\sigma}\in M(B)\otimes H$ satisfying that
$$
w^{\sigma *}((1\rtimes_{\underline{\sigma}}e)\otimes 1)w^{\sigma}=\widehat{\underline{\sigma}}
(1\rtimes_{\underline{\sigma}}e) , 
$$Let $U^{\sigma}=w^{\sigma}(z^{\sigma*}\otimes 1)$,
where $z^{\sigma}=(\id_{M(B)}\otimes \epsilon)(w^{\sigma})\in M(B)^{\underline{\sigma}}$.
Then $U^{\sigma}\in M(B)\otimes H$ and
satisfies that
$$
\widehat{U}^{\sigma}(1^0 )=1 , \quad
\widehat{U}^{\sigma}(\phi_{(1)})a\widehat{U}^{\sigma *}(\phi_{(2)})
\in M(B)^{\underline{\sigma}} 
$$
for any $a\in M(B)^{\underline{\sigma}}$, $\phi\in H^0$.
Let $\widehat{u}^{\sigma}$ be a bilinear map from $H^0 \times H^0$ to $M(B)$
defined by
$$
\widehat{u}^{\sigma}(\phi, \psi)=\widehat{U}^{\sigma}(\phi_{(1)})
\widehat{U}^{\sigma}(\psi_{(1)})\widehat{U}^{\sigma *}(\phi_{(2)}\psi_{(2)})
$$
for any $\phi, \psi\in H^0$. Then by \cite [Lemma 5.4]{KT1:inclusion},
$\widehat{u}^{\sigma}(\phi, \psi)\in M(B)^{\underline{\sigma}}$
for any $\phi, \psi\in H^0$ and by \cite [Corollary 5.3]{KT1:inclusion}, the map
$$
H^0 \times M(B)^{\underline{\sigma}}\longrightarrow M(B)^{\underline{\sigma}}: (\phi, a)\mapsto
\widehat{U}^{\sigma}(\phi_{(1)})a\widehat{U}^{\sigma *}(\phi_{(2)})
$$
is a weak action of $H^0$ on $M(B)^{\underline{\sigma}}$. Furthermore, by \cite [Proposition 5.6]{KT1:inclusion}
$\widehat{u}^{\sigma}$ is a unitary cocycle for the above weak action.
Let $u^{\sigma}$ be the unitary element in $M(B)^{\underline{\sigma}}\otimes H\otimes H$ induced by
$\widehat{u}^{\sigma}$ and $\rho'$ the weak coaction of $H$ on $M(B)^{\underline{\sigma}}$
induced by the above weak action. Thus we obtain a twisted coaction $(\rho', u^{\sigma})$
of $H$ on $M(B)^{\underline{\sigma}}$. Let $\pi'$ be the map from
$M(B)^{\underline{\sigma}}\rtimes_{\rho' , u^{\sigma}}H^0$ to $M(B)$ defined by
$$
\pi'(a\rtimes_{\rho' , u^{\sigma}}\phi)=a\widehat{U}^{\sigma}(\phi)
$$
for any $a\in M(B)^{\underline{\sigma}}$, $\phi\in H^0$. Then by
\cite [Proposition 6.1 and Theorem 6.4]{KT1:inclusion}, $\pi'$ is an isomorphism of
$M(B)^{\underline{\sigma}}\rtimes_{\rho' , u^{\sigma}}H^0$ onto $M(B)$ satisfying that
$$
\underline{\sigma}\circ\pi' =(\pi' \otimes\id_{H^0} )\circ\widehat{\rho'} \, , \quad
E_1^{\rho' , u^{\sigma}}=E^{\underline{\sigma}}\circ\pi' ,
$$
where $E_1^{\rho' , u^{\sigma}}$ is the canonical conditional expectation
from $M(B)^{\underline{\widehat{\sigma}}}\rtimes_{\rho' , u^{\sigma}}H^0$ onto $M(B)^{\underline{\sigma}}$
and $E^{\underline{\sigma}}$
is the canonical conditional expectation from $M(B)$ onto $M(B)^{\underline{\sigma}}$.
Let $\rho=\rho' |_{B^{\sigma}}$.

\begin{lemma}\label{lem:restriction1}With the above notations, $(\rho, u^{\sigma})$ is a twisted
coaction of $H$ on $B^{\sigma}$ and $\underline{\rho}=\rho'$.
\end{lemma}
\begin{proof}By the definition of $\rho$, for any $a\in B^{\sigma}$,
$$
\rho(a)=U^{\sigma}(a\otimes 1)U^{\sigma *}.
$$
Since $a\in B^{\sigma}\subset M(B^{\sigma})=M(B)^{\underline{\sigma}}$
by Lemma \ref{lem:fix}, $\rho(a)\in M(B)^{\underline{\sigma}}\otimes H$.
On the other hand, since $U^{\sigma}\in M(B)\otimes H$, $\rho(a)\in B\otimes H$.
Thus $\rho(a)\in (M(B)^{\underline{\sigma}}\otimes H)\cap(B\otimes H)=B^{\sigma}\otimes H$.
Hence $\rho$ is a homomorphism of $B^{\sigma}$ to $B^{\sigma}\otimes H$. Since
$(\rho' \otimes\id)\circ\rho' =\Ad (u^{\sigma})\circ(\id\otimes \Delta)\circ\rho'$ and $\rho(a)\in B^{\sigma}\otimes H$ for
any $a\in B^{\sigma}$, we can see that $(\rho\otimes\id)\circ\rho=\Ad(u^{\sigma})\circ(\id\otimes\Delta)\circ\rho$.
By the definition of $\rho'$, $\rho'$ is strictly continuous on $M(B)^{\underline{\sigma}}$. Hence for any approximate unit
$\{u_{\alpha}\}$ of $B^{\sigma}$,
$$
1\otimes 1=\rho' (1)=\rho' (\lim_{\alpha\to\infty}u_{\alpha})=\lim_{\alpha\to\infty}\rho' (u_{\alpha})
=\lim_{\alpha\to\infty}\rho(u_{\alpha}) ,
$$
where the limits are taken under the strict topologies in $M(B^{\sigma})$ and
$M(B^{\sigma})\otimes H$, respectively. This means that
$$
\overline{\rho(B^{\sigma})(B^{\sigma}\otimes H)}=B^{\sigma}\otimes H .
$$
It follows that $(\rho, u^{\sigma})$ is a twisted coaction of $H$ on $B^{\sigma}$.
Furthermore, since $\rho'$ is strictly continuous, $\rho'=\underline{\rho}$ on $M(B^{\sigma})$.
\end{proof}

Let $\pi=\pi' |_{B^{\sigma}\rtimes_{\rho, u^{\sigma}}H^0 }$.

\begin{lemma}\label{lem:restriction2}With the above notations, $\pi$ is an isomorphism of
$B^{\sigma}\rtimes_{\rho, u^{\sigma}}H^0$ onto $B$ satisfying that
$$
\sigma\circ\pi=(\pi\otimes\id_{H^0} )\circ\widehat{\rho} , \quad
E_1^{\rho, u^{\sigma}}=E^{\sigma}\circ\pi ,
$$
where $E_1^{\rho, u^{\sigma}}$ is the canonical conditional expectation
from $B^{\sigma}\rtimes_{\rho, u^{\sigma}}H^0$ onto $B^{\sigma}$ and $E^{\sigma}$ is the canonical
conditional expectation from $B$ onto $B^{\sigma}$. Furthermore, $\pi'=\underline{\pi}$.
\end{lemma}
\begin{proof}Let $E^{\underline{\sigma}}$ be the canonical conditional expectation
from $M(B)$ onto $M(B)^{\underline{\sigma}}$. By \cite [Proposition 4.3 and Remark 4.9]
{KT1:inclusion},
$\{(\sqrt{f_k}\widehat{U}^{\sigma}(\omega_{ij}^k )^*, \, \sqrt{f_k}\widehat{U}^{\sigma}(\omega_{ij}^k ))\}_{i, j, k}$
is a quasi-basis for $E^{\underline{\sigma}}$. Hence for any $b\in B$,
$$
b=\sum_{i, j, k}f_k E^{\underline{\sigma}}(b\widehat{U}^{\sigma}(\omega_{ij}^k )^* )
\widehat{U}^{\sigma}(\omega_{ij}^k ) .
$$
Since $\widehat{U}^{\sigma}(\omega_{ij}^k )\in M(B)$ for any $i, j, k$,
$$
E^{\underline{\sigma}}(b\widehat{U}^{\sigma}(\omega_{ij}^k )^* )\in B^{\sigma}
$$
for any $i, j, k$ and $b\in B$. Let $a=\sum_{i, j, k}f_k E^{\underline{\sigma}}(b\widehat{U}^{\sigma}(\omega_{ij}^k )^* )
\rtimes_{\rho, u^{\sigma}}\omega_{ij}^k$. Then $a\in B^{\sigma}\rtimes_{\rho, u^{\sigma}}H^0$
and $\pi(a)=b$. Thus $\pi$ is surjective. Since $\pi'$ is an isomorphism
of $M(B)^{\underline{\sigma}}\rtimes_{\underline{\rho}, u^{\sigma}}H^0$ onto
$M(B)$, we can see that $\pi$ is an isomorphism of $B^{\sigma}\rtimes_{\rho, u^{\sigma}}H^0$ onto $B$.
Also, since $\underline{\sigma}\circ\pi' =(\pi' \otimes\id)\circ\widehat{\underline{\rho}}$ and
$E_1^{\underline{\rho}, u^{\sigma}}=E^{\underline{\sigma}}\circ\pi' $,
we can see that
$$
\sigma\circ\pi=(\pi\otimes\id)\circ\widehat{\rho}, \quad
E_1^{\rho, u^{\sigma}}\circ\pi .
$$
Furthermore, by the definition of $\pi'$, $\pi'$ is strictly continuous. Thus $\pi' =\underline{\pi}$.
\end{proof}

Combining Lemmas \ref{lem:fix}, \ref{lem:restriction1} and \ref{lem:restriction2}, we obtain the following proposition:

\begin{prop}\label{prop:nonunital2}Let $B$ be a $C^*$-algebra and $\sigma$ a coaction of $H^0$
on $B$. We suppose that $\widehat{\underline{\sigma}}(1\rtimes_{\underline{\sigma}}e)
\sim(1\rtimes_{\underline{\sigma}}e)\otimes 1$ in $(M(B)\rtimes_{\underline{\sigma}}H)\otimes H$.
Then there are a twisted coaction $(\rho, u^{\sigma})$ of $H$ on $B^{\sigma}$ and an isomorphism $\pi$ of
$B^{\sigma}\rtimes_{\rho, u^{\sigma}}H^0$ onto $B$ satisfying that
$$
\sigma\circ\pi =(\pi\otimes \id_H )\circ\widehat{\rho}, \quad E_1^{\rho, u^{\sigma}}=E^{\sigma}\circ\pi ,
$$
where $B^{\sigma}$ is the fixed point $C^*$-subalgebra of $B$ for $\sigma$ and $E_1^{\rho, u^{\sigma}}$
and $E^{\sigma}$ are the canonical conditional expectations from $B$ and $B^{\sigma}\rtimes_{\rho, u^{\sigma}}H^0$
onto $B^{\sigma}$, respectively.
\end{prop}

\section{Twisted coactions on a Hilbert $C^*$-bimodule and
strong Morita equivalence for twisted coactions}\label{sec:morita}

First, we shall define crossed products of Hilbert $C^*$-bimodules in the
sense of Brown, Mingo and Shen \cite {BMS:quasi} and show their
duality theorem, which is similar to \cite [Theorem 5.7]{KT3:equivalence}.
We give the definition of a Hilbert $C^*$-bimodule in the sense of
\cite {BMS:quasi}.
\par
Let $A$ and $B$ be $C^*$-algebras. Let $X$ be a left pre-Hilbert $A$-bimodule
and a right pre-Hilbert $B$-module. Its left $A$-valued inner product
and right $B$-valued inner product denote by $ {}_A \la \cdot \, , \, \cdot \ra$
and $\la \cdot \, , \, \cdot \ra_B$, respectively.

\begin{Def}\label{Def:bimodule}We call $X$ a
\sl
pre-Hilbert $A-B$-bimodule
\rm
if $X$ satisfies the condition
$$
{}_A \la x , \, y \ra z=x \la y, \, z \ra_B
$$
for any $x, y, z\in X$. We call $X$ a
\sl
Hilbert $A-B$-bimodule
\rm
if  $X$
is complete with the norms.
\end{Def}

\begin{remark}\label{remark:bimodule}We suppose that $X$ is a pre-Hilbert $A-B$-bimodule. Then
by \cite [Remark 1.9]{BMS:quasi}, we can see the following:
\newline
(1) For any $x\in X$, $||{}_A \la x, \, x\ra ||=||\la x, \, x \ra_B ||$,
\newline
(2) For any $a\in A$, $b\in B$ and $x, y\in X$,
$$
{}_A \la x, \, yb \ra ={}_A \la xb^* , \, y \ra , \quad
\la ax, \, y \ra_B = \la x, \, a^* y \ra_B .
$$
(3) If $X$ is complete with the norm and full with the both-sided inner products, then
$X$ is an $A-B$-equivalence bimodule.
\end{remark}

In this paper, by the words `` pre-Hilbert $C^*$-bimodules" and `` Hilbert $C^*$-bimodules",
we mean pre-Hilbert
$C^*$-bimodule and Hilbert $C^*$-bimodules in the sense of \cite {BMS:quasi}, respectively.
\par
Let $A$ and $B$ be $C^*$-algebras. Let $X$
be a Hilbert $A-B$-bimodule and let $\BB_B (X)$ be the $C^*$-algebra
of all right $B$-linear operators on $X$ for which
there is a right adjoint $B$-linear operator on $X$. We note that a right $B$-linear operator
on $X$ is bounded. For each $x, y\in X$, let $\theta_{x, y}$ be a rank-one
operator on $X$ defined by $\theta_{x, y}(z)=x\la y, z \ra_B$ for any $z\in X$.
Then $\theta_{x, y}$ is a right $B$-linear operator on $X$. Let $\BK_B (X)$ be the
closure of all linear spans of such $\theta_{x, y}$. Then $\BK_B (X)$ is a closed two-sided
ideal of $\BB_B (X)$. Similarly, we define ${}_A \BB(X)$ and ${}_A \BK(X)$.
If $X$ is an $A-B$-equivalence bimodule,
we identify $A$ and $M(A)$ with $\BK_B (X)$ and $\BB_B (X)$,
respectively and identify $B$
and $M(B)$ with ${}_A \BB(X)$ and ${}_A \BK(X)$, respectively.
For any $a\in M(A)$, we regard $a\in M(A)$ as an element in $\BB_B (X)$ as follows: For
any $b\in A$, $x\in X$,
$$
a(bx)=(ab)x.
$$
Since $X=\overline{AX}$ by \cite [Proposition 1.7]{BMS:quasi}, we can obtain an element in $\BB_B (X)$
induced by $a\in M(A)$. Similarly, we can obtain an element in ${}_A \BB (X)$ induced by
any $b\in M(B)$.

\begin{lemma}\label{lem:convergent}With the above notations, we suppose that $X$ is
a Hilbert $A-B$-bimodule. For any $a\in M(A)$, there is
a bounded net $\{a_{\alpha}\}_{\alpha\in\Gamma}\subset A$ such that $ax=\lim_{\alpha\to\infty}a_{\alpha}x$ for any $x\in X$.
\end{lemma}
\begin{proof}Since$a\in M(A)$, there is a bounded net $\{a_{\alpha}\}_{\alpha\in\Gamma}\subset A$
such that $\{a_{\alpha}\}_{\alpha\in\Gamma}$ converges to $a$ strictly.
We can prove that $ax=\lim_{\alpha\to\infty}a_{\alpha}x$ for any
$x\in X$ in a routine way since $X=\overline{AX}$ by \cite [Proposition 1.7]{BMS:quasi}.
\end{proof}

Let $(\rho, u)$ and $(\sigma, v)$ be twisted coactions of $H^0$ on $A$ and $B$,
respectively.

\begin{Def}\label{Def:coaction}Let $\lambda$ be a linear map from
a Hilbert $A-B$-bimodule $X$ to $X\otimes H^0$.
Then we say that $\lambda$ is a
\sl
twisted coaction
\rm
of $H^0$ on $X$ with respect to
$(A, B, \rho, u, \sigma, v)$ if the following conditions hold:
\newline
(1) $\lambda(ax)=\rho(a)\lambda(x)$ for any $a\in A$, $x\in X$,
\newline
(2) $\lambda(xb)=\lambda(x)\sigma(b)$ for any $b\in B$, $x\in X$,
\newline
(3) $\rho({}_A \la x, y \ra)={}_{A\otimes H^0 }\la \lambda(x), \lambda(y) \ra$ for ny $x, y\in X$,
\newline
(4) $\sigma(\la x, y \ra_B )=\la \lambda(x), \lambda(y) \ra_{B\otimes H^0}$ for any $x, y\in X$,
\newline
(5) $(\id_X \otimes\epsilon^0 )\circ\lambda=\id_X$,
\newline
(6) $(\lambda\otimes\id)(\lambda(x))=u(\id\otimes\Delta^0 )(\lambda(x))v^* $ for any $x\in X$,
\newline
where $u$ and $v$ are regarded as elements in $\BB_B (X)$ and ${}_A \BB(X)$, respectively.
\end{Def}

We note that the twisted coaction $\lambda$ of $H^0$ on
the Hilbert $A-B$-bimodule $X$ with respect to $(A, B, \rho, u, \sigma, v)$ is isometric.
Indeed, for any $x\in X$
$$
||\lambda(x)||^2 = ||{}_{A\otimes H^0} \la \lambda(x), \lambda(x) \ra ||
=||\rho({}_A \la x, y \ra) ||=||{}_A \la x, y \ra ||=||x||^2.
$$

Let $\lambda$ be a twisted coaction of $H^0$on a Hilbert $A-B$-bimodule $X$ with respect to
$(A, B, \rho, u, \sigma, v )$. We define the
\sl
twisted action
\rm
of $H$ on $X$ induced by $\lambda$ as
follows: For any $x\in X$, $h\in H$,
$$
h\cdot_{\lambda}x =(\id\otimes h)(\lambda(x))=\lambda(x)^{\widehat{}}(h) ,
$$
where $\lambda(x)^{\widehat{}}$ is the element in $\Hom (H, X)$ induced by $\lambda(x)$
in $X\otimes H^0$. Then we obtain the following conditions which are equivalent to
Conditions (1)-(6) in Definition \ref{Def:coaction}, respectively:
\newline
(1)' $h\cdot_{\lambda}ax =[h_{(1)}\cdot_{\rho, u}a][h_{(2)}\cdot_{\lambda}x]$ for any $a\in A$, $x\in X$,
\newline
(2)' $h\cdot_{\lambda}xb =[h_{(1)}\cdot_{\lambda}x][h_{(2)}\cdot_{\sigma, v}b]$ for any $b\in B$, $x\in X$,
\newline
(3)' $h\cdot_{\rho} {}_A \la x, y \ra ={}_A \la [h_{(1)}\cdot_{\lambda}x], \, [S(h_{(2)}^* )\cdot_{\lambda}y] \ra$
for any $x, y\in X$,
\newline
(4)' $h\cdot_{\sigma} \la x, y \ra_B =\la [S(h_{(1)}^* )\cdot_{\lambda}x], \, [h_{(2)}\cdot_{\lambda}y] \ra_B$
for any $x, y\in X$,
\newline
(5)' $1_H \cdot_{\lambda}x =x$ for any $x\in X$,
\newline
(6)' $h\cdot_{\lambda}[l\cdot_{\lambda}x]
=\widehat{u}(h_{(1)}, l_{(1)})[h_{(2)}l_{(2)}\cdot_{\lambda}x]\widehat{v}^* (h_{(3)}, l_{(3)})$
for any $x\in X$, $h, l\in H$,
\newline
where $\widehat{u}$ and $\widehat{v}$ are elements in $\Hom (H\times H, M(A))$ and $\Hom (H\times H, M(B))$
induced by $u\in M(A)\otimes H^0 \otimes H^0$ and $v\in M(B)\otimes H^0 \otimes H^0$, respectively.

\begin{remark}\label{remark:coaction}In Definition \ref{Def:coaction}, if $\rho$ and $\sigma$ are coactions
of $H^0$ on $A$ and $B$, respectively, then Condition (6) in Definition \ref{Def:coaction} and its equivalent Condition (6)' are
following, respectively:
\newline
$(6)$ $(\lambda\otimes\id)\circ\lambda=(\id\otimes\Delta^0 )\circ\lambda$,
\newline
$(6)$' $h\cdot_{\lambda}[l\cdot_{\lambda}x]=hl\cdot_{\lambda}x$ for any $x\in X$.
\newline
In this case, we call $\lambda$ a {\it coaction} of $H^0$ on $X$ with respect to $(A, B, \rho, \sigma)$.
\end{remark}
\par
Next, we shall define crossed products of Hilbert $C^*$-bimodules by twisted coactions in the same
way as in \cite [Section 4]{KT3:equivalence} and give a duality theorem for them.
\par
Let $(\rho, u)$ and $(\sigma, v)$ be twisted coactions of $H^0$ on $C^*$-algebras $A$ and
$B$, respectively. Let $\lambda$ be a twisted coaction of $H^0$ on a Hilbert $A-B$-bimodule
$X$ with respect to $(A, B, \rho, u, \sigma, v )$. We define $X\rtimes_{\lambda}H$, a
Hilbert $A\rtimes_{\rho, u}H-B\rtimes_{\sigma, v}H$-bimodule as follows:
Let $(X\rtimes_{\lambda}H)_0$ be just $X\otimes H$ (the algebraic tensor product) as vector spaces.
Its left and right actions are given by
\begin{align*}
(a\rtimes_{\rho, u}h)(x\rtimes_{\lambda}l) & =a[h_{(1)}\cdot_{\lambda}x]\widehat{v}(h_{(2)}, l_{(1)})
\rtimes_{\lambda}h_{(3)}l_{(2)} , \\
(x\rtimes_{\lambda}l)(b\rtimes_{\sigma, v}m) & =x[l_{(1)}\cdot_{\sigma, v}b]\widehat{v}(l_{(2)}, m_{(1)})
\rtimes_{\lambda}l_{(3)}m_{(2)}
\end{align*}
for any $a\in A, b\in B, x\in X$ and $h, l, m\in H$. Also, its left $A\rtimes_{\rho, u}H$-valued
and right $B\rtimes_{\sigma,v}H$-valued inner products are given by
\begin{align*}
{}_{A\rtimes_{\rho, u}H} \la x\rtimes_{\lambda}h, \, y\rtimes_{\lambda}l \ra & ={}_A \la x, \,
[S(h_{(2)}l_{(3)}^* )^* \cdot_{\lambda}y ]\widehat{v}(S(h_{(1)}l_{(2)}^* )^* , \, l_{(1)}) \ra
\rtimes_{\rho, u}h_{(3)}l_{(4)}^* , \\
\la x\rtimes_{\lambda}h, \, y\rtimes_{\lambda}l \ra_{B\rtimes_{\sigma, v}H} & =
\widehat{v}^* (h_{(2)}^* , \, S(h_{(1)})^* )[h_{(3)}^* \cdot_{\sigma, v}\la x, y \ra_B ]\widehat{v}(h_{(4)}^* ,l_{(1)})
\rtimes_{\sigma, v}h_{(5)}^* l_{(2)}
\end{align*}
for any $x, y\in X$ and $h, l\in H$. In the same way as in \cite [Section 4]{KT3:equivalence}, we can
see that $(X\rtimes_{\lambda}H)_0$ is a pre-Hilbert $A\rtimes_{\rho, u}H-B\rtimes_{\sigma, v}H$-bimodule.
Let $X\rtimes_{\lambda}H$ be the completion of $(X\rtimes_{\lambda}H)_0$.
It is a Hilbert $A\rtimes_{\rho, u}H-B\rtimes_{\sigma, v}H$-bimodule. Let $\widehat{\lambda}$
be a linear map from $(X\rtimes_{\lambda}H)_0$ to $(X\rtimes_{\lambda}H)_0 \otimes H$
defined by
$$
\widehat{\lambda}(x\rtimes_{\lambda}h)=(x\rtimes_{\lambda}h_{(1)})\otimes h_{(2)}
$$
for any $x\in X$, $h\in H$. By easy computations, we can see that $\widehat{\lambda}$ is a linear
map from $H$ to $(X\rtimes_{\lambda}H)_0 \otimes H$ satisfying Conditions (1)-(6)
in Definition \ref{Def:coaction}.
Thus for any $x\in (X\rtimes_{\lambda}H)_0$,
$$
||\widehat{\lambda}(x)||^2 =||{}_{(A\rtimes_{\rho. u}H)\otimes H} \la \widehat{\lambda}(x), \, \widehat{\lambda}(x) \ra ||
=||\widehat{\rho}({}_A \la x, \, x \ra)||=||{}_A \la x, x \ra ||=||x||^2 .
$$
Hence $\widehat{\lambda}$ is an isometry. We extend $\widehat{\lambda}$ to $X\rtimes_{\lambda}H$.
We can see that the extension of $\widehat{\lambda}$ is a coaction of $H$ on
$X\rtimes_{\lambda}H$ with respect to $(A\rtimes_{\rho, u}H, B\rtimes_{\sigma, v}H, \widehat{\rho}, \widehat{\sigma})$.
We also denote it by the same symbol $\widehat{\lambda}$ and call it
the dual coaction of $\lambda$. Similarly we define the second dual coaction of $\lambda$,
which is a coaction of $H^0$ on $X\rtimes_{\lambda}H\rtimes_{\widehat{\lambda}}H^0$.
Let $\Lambda$ be as in Section \ref{sec:pre}. For any $I=(i, j, k)\in \Lambda$, let  $W_I^{\rho}$, $V_I^{\rho}$
be elements in $M(A)\rtimes_{\underline{\rho}, u}H\rtimes_{\underline{\widehat{\rho}}}H^0$ defined by
$$
W_I^{\rho}=\sqrt{d_k}\rtimes_{\underline{\rho}, u}w_{ij}^k, \quad
V_I^{\rho}=(1\rtimes_{\underline{\rho}, u}1\rtimes_{\widehat{\underline{\rho}}}\tau)(W_I^{\rho}
\rtimes_{\widehat{\underline{\rho}}}1^0 ) .
$$
Similarly for any $I=(i, j, k)\in\Lambda$, we define elements
$$
W_I^{\sigma}=\sqrt{d_k}\rtimes_{\underline{\sigma}, v}w_{ij}^k, \quad
V_I^{\sigma}=(1\rtimes_{\underline{\sigma}, v}1\rtimes_{\widehat{\underline{\sigma}}}\tau)(W_I^{\sigma}
\rtimes_{\widehat{\underline{\sigma}}}1^0 )
$$
in $M(B)\rtimes_{\underline{\sigma}, v}H\rtimes_{\widehat{\underline{\sigma}}}H^0$.
We regard $M_N (\BC)$ as an equivalence $M_N (\BC)-M_N (\BC)$-bimodule in the
usual way. Let $X\otimes M_N (\BC)$ be the exterior tensor product of $X$ and $M_N (\BC)$,
which is a Hilbert $A\otimes M_N (\BC)-B\otimes M_N (\BC)$-bimodule.
Let $\{f_{IJ}\}_{I, J\in\Lambda}$ be a system of matrix units of $M_N (\BC)$.
Let $\Psi_X$ be a linear map from $X\otimes M_N (\BC)$ to $X\rtimes_{\lambda}H\rtimes_{\widehat{\lambda}}H^0$
defined by
$$
\Psi_X (\sum_{I, J}x_{IJ}\otimes f_{IJ})=\sum_{I, J}V_I^{\rho*}(x_{IJ}\rtimes_{\lambda}1\rtimes_{\widehat{\lambda}}1^0 )
V_J^{\sigma} .
$$
Let $\Psi_A$ and $\Psi_B$ be the isomorphisms of $A\otimes M_N (\BC)$ and $B\otimes M_N (\BC)$
onto $A\rtimes_{\rho, u}H\rtimes_{\widehat{\rho}}H^0$ and $B\rtimes_{\sigma, v}H\rtimes_{\widehat{\sigma}}H^0$
defined in Proposition \ref{prop:nonunital}, respectively. Then we have the same lemmas as
\cite [Lemmas 5.1 and 5.5]{KT3:equivalence}. Hence $\Psi_X$ is an isometry from $X\otimes M_N (\BC)$
to $X\rtimes_{\lambda}H\rtimes_{\widehat{\lambda}}H^0$ whose image is
$(X\rtimes_{\lambda}H)_0 \rtimes_{\widehat{\lambda}}H^0$, the linear span of the set
$$
\{x\rtimes_{\lambda}h\rtimes_{\widehat{\lambda}}\phi \, | \, x\in X, h\in H, \phi\in H^0 \} .
$$
Since $X\otimes M_N (\BC)$ is complete, so is $(X\rtimes_{\lambda}H)_0 \rtimes_{\widehat{\lambda}}H^0$.
Furthermore, we claim that $(X\rtimes_{\lambda}H)_0$ is also complete. In order to show it,
we need the following lemma: Let $E_1^{\lambda}$ be a linear map from $(X\rtimes_{\lambda}H)_0$
onto $X$ defined by
$$
E_1^{\lambda}(x\rtimes_{\lambda}h)=\tau(h)x
$$
for any $x\in X$, $h\in H$.

\begin{lemma}\label{lem:continuous}With the above notations, $E_1^{\lambda}$ is continuous.
\end{lemma}
\begin{proof}In the same way as in the proof of \cite [Lemma 5.6]{KT3:equivalence}, we can see that
$$
E_1^{\lambda}(x\rtimes_{\lambda}h)=\tau\cdot_{\widehat{\lambda}}(x\rtimes_{\lambda}h)
=\widehat{V}^{\widehat{\rho}}(\tau_{(1)})(x\rtimes_{\lambda}h\rtimes_{\widehat{\lambda}}1^0 )
\widehat{V}^{\widehat{\sigma}*}(\tau_{(2)}) ,
$$
where we identify $X\rtimes_{\lambda}H\rtimes_{\widehat{\lambda}}1^0 $
with $X\rtimes_{\lambda}H$ and
$$
\widehat{V}^{\widehat{\rho}}(\phi)=1\rtimes_{\underline{\rho}, u}1\rtimes_{\underline{\widehat{\rho}}}\phi, \quad
\widehat{V}^{\widehat{\sigma}}(\phi)=1\rtimes_{\underline{\sigma}, v}1\rtimes_{\underline{\widehat{\sigma}}}\phi
$$
for any $\phi\in H^0$. Hence $E_1^{\lambda}$ is continuous.
\end{proof}
Let $E_2^{\lambda}$ be a linear map from $(X\rtimes_{\lambda}H\rtimes_{\widehat{\lambda}}H^0 )_0$
to $X\rtimes_{\lambda}H$ defined by
$$
E_2^{\lambda}(x\rtimes_{\widehat{\lambda}}\phi)=\phi(e)x
$$
for any $x\in X\rtimes_{\lambda}H$, $\phi\in H^0$.

\begin{lemma}\label{lem:complete}With the above notations,
$(X\rtimes_{\lambda}H)_0$ is complete.
\end{lemma}
\begin{proof}Let $\{x_n \}$ be a Cauchy sequence in $(X\rtimes_{\lambda}H)_0$.
Using Lemma \ref{lem:continuous} and the lonear map $E_2^{\lambda}$, we can see that
$\{x_n \}$ is convergent in $(X\rtimes_{\lambda}H)_0$.
\end{proof}

By Lemma \ref{lem:complete}, $X\rtimes_{\lambda}H=(X\rtimes_{\lambda}H)_0$.
In the same way as in the proof of \cite [Theorem 5.7]{KT3:equivalence}, we obtain the
following proposition using Lemma \ref{lem:complete}:

\begin{prop}\label{prop:dual2}Let $A$, $B$ be $C^*$-algebras and $H$ a finite dimensional
$C^*$-Hopf algebra with its dual $C^*$-Hopf algebra $H^0$. Let $(\rho, u)$ and $(\sigma, v)$
be twisted coactions of $H^0$ on $A$ and $B$, respectively. Let $\lambda$ be a twisted
coaction of $H^0$ on a Hilbert $A-B$-bimodule $X$ with respect to $(A, B, \rho, u, \sigma, v)$.
Then there is an isomorphism $\Psi_X$ from $X\otimes M_N (\BC)$ onto
$X\rtimes_{\lambda}H\rtimes_{\widehat{\lambda}}H^0$ satisfying that
\begin{align*}
(1)\, \Psi_X ((\sum_{I, J}a_{IJ}\otimes f_{IJ})(\sum_{I, J}x_{IJ}\otimes f_{IJ}))
& =\Psi_A (\sum_{I, J}a_{IJ}\otimes f_{IJ})\Psi_X (\sum_{I, J}x_{IJ}\otimes f_{IJ}) , \\
(2)\, \Psi_X ((\sum_{I, J}x_{IJ}\otimes f_{IJ})(\sum_{I, J}b_{IJ}\otimes f_{IJ}))
& =\Psi_X (\sum_{I, J}x_{IJ}\otimes f_{IJ})\Psi_B (\sum_{I, J}b_{IJ}\otimes f_{IJ}) , \\
(3)\, {}_{A\rtimes_{\rho, u}H\rtimes_{\widehat{\rho}}H^0} \la \Psi_X (\sum_{I, J}x_{IJ}\otimes f_{IJ}) \, , \,
\Psi_X &  (\sum_{I, J}  y_{IJ}\otimes f_{IJ}) \ra \\
= \Psi_A ({}_{A\otimes M_N (\BC)} & \la \sum_{I, J}x_{IJ}\otimes f_{IJ}\, , \,
\sum_{I, J}y_{IJ}\otimes f_{IJ} \ra ) , \\
(4) \, \, \la \Psi_X (\sum_{I, J}x_{IJ}\otimes f_{IJ}) \, , \,
\Psi_X (\sum_{I, J} y_{IJ}\otimes & f_{IJ}) \ra_{B\rtimes_{\sigma, v}H\rtimes_{\widehat{\sigma}}H^0} \\
= \Psi_B ( \la & \sum_{I, J}x_{IJ}\otimes f_{IJ}\, , \,
\sum_{I, J}y_{IJ}\otimes f_{IJ} \ra_{B\otimes M_N (\BC)} )
\end{align*}
for any $a_{IJ}\in A$, $b_{IJ}\in B$, $x_{IJ}\, , \, y_{IJ}\in X$, $I, J\in \Lambda$,
where $X\rtimes_{\lambda}H\rtimes_{\widehat{\lambda}}H^0$
is a Hilbert $A\rtimes_{\rho, u}H\rtimes_{\widehat{\rho}}H^0 -B\rtimes_{\sigma, v}H\rtimes_{\widehat{\sigma}}H^0$-
bimodule and $X\otimes M_N (\BC)$ is an exterior tensor product of $X$ and the Hilbert
$M_N (\BC)-M_N (\BC)$-bimodule $M_N (\BC)$. Furthermore, there are unitary elements
$U\in (M(A)\rtimes_{\underline{\rho}, u}H\rtimes_{\widehat{\underline{\rho}}}H^0 )\otimes H^0$
and $V\in (M(B)\rtimes_{\underline{\sigma}, v}H\rtimes_{\widehat{\underline{\sigma}}}H^0 )\otimes H^0$
such that
$$
U\widehat{\widehat{\lambda}}(x)V=((\Psi_X \otimes\id)\circ(\lambda\otimes\id_{M_N (\BC)})\circ\Psi_X^{-1})(x)
$$
for any $x\in X\otimes M_N (\BC)$.
\end{prop}

The above proposition has already obtained in the case of
Kac systems by Guo and Zhang \cite {GZ:Kac2}, which is a generalization of
the above result. Also, we have the following lemmas:

\begin{lemma}\label{lem:full0}With the above notations, if $X$ is full with the both-sided inner products,
then so is $X\rtimes_{\lambda}H$.
\end{lemma}
\begin{proof}Modifying the proof of \cite [Lemma 4.5]{KT3:equivalence},
we can prove the lemma.
\end{proof}
\begin{lemma}\label{lem:full}With the above notations, if $X\rtimes_{\lambda}H$ is full with
the both-sided inner products, then so is $X$.
\end{lemma}
\begin{proof}
Since $X\rtimes_{\lambda}H$ is full with the both-sided inner products,
so is $X\rtimes_{\lambda}H\rtimes_{\widehat{\lambda}}H^0$ by Lemma \ref{lem:full0}.
Thus $X\otimes M_N (\BC)$ is full with the both-sided inner products by
Proposition \ref {prop:dual2}. Let $f$ be a minimal projection in $M_N (\BC)$.
Then
\begin{align*}
A\otimes f &= (1_{M(A)}\otimes f)(A\otimes M_N (\BC))(1_{M(A)}\otimes f) \\
& =(1\otimes f) \, \overline{{}_{A\otimes M_N (\BC)} \la X\otimes M_N (\BC), \, X\otimes M_N (\BC) \ra }(1\otimes f) \\
& =\overline{{}_A \la X, \, X \ra\otimes f M_N (\BC) f}=\overline{{}_A \la X, \, X \ra }\otimes f .
\end{align*}
Hence $X$ is full with the left-sided inner product. Similarly, we can see that $X$ is full with the right-sided
inner product. Therefore, we obtain the conclusion.
\end{proof}

\begin{Def}\label{Def:morita}Let $(\rho, u)$ and $(\sigma, v)$ be twisted coactions of $H^0$ on
$C^*$-algebras $A$ and $B$, respectively. Then $(\rho, u)$ is
\sl
strongly Morita equivalent
\rm
to $(\sigma, v)$ if there are an $A-B$-equivalence bimodule $X$ and a twisted coaction $\lambda$
of $H^0$ on $X$ with respect to $(A, B, \rho, u, \sigma, v )$.
\end{Def}

In the same way as in \cite [Section 3]{KT3:equivalence}, we can see that the strong Morita
equivalence for twisted coactions of $H^0$ on $C^*$-algebras is an equivalence relation.
Also, we can obtain the following lemma in the similar way to \cite [Lemma 3.12]{KT3:equivalence}
using approximate units in a $C^*$-algebra. We give it without its proof.

\begin{lemma}\label{lem:extmorita}Let $(\rho, u)$ and $(\sigma, v)$ be
twisted coactions of $H^0$ on $A$. Then the following conditions are equivalent:
\newline
$(1)$ The twisted coactions $(\rho, u)$ and $(\sigma,v)$ are exterior equivalent,
\newline
$(2)$ The twisted coactions $(\rho, u)$ and $(\sigma,v)$ are strongly Morita equivalent by
a twisted coaction $\lambda$ of $H^0$ on ${}_A A_A$, which is a linear map from ${}_A A_A$
to ${}_{A\otimes H^0} A\otimes H_{A\otimes H^0}^0$, where ${}_A A_A$
and ${}_{A\otimes H^0} A\otimes H_{A\otimes H^0}^0$ are regarded as an $A-A$-equivalence
bimodule and an $A\otimes H^0 -A\otimes H^0$-equivalence bimodule in the
usual way.
\end{lemma}

\begin{remark}\label{remark:right}Let $A$ and $B$ be $C^*$-algebras and $\sigma$ a coaction of $H^0$ on $B$.
Let $X$ be an $A-B$-equivalence bimodule and $\lambda$ a linear map from $X$
to $X\otimes H^0$ satisfying that
\newline
(1) $\lambda(xb)=\lambda(x)\sigma(b)$ for any $b\in B$, $x\in X$,
\newline
(2) $\sigma(\la x, y \ra_B )=\la \lambda(x) , \, \lambda(y) \ra_{B\otimes H^0 }$ for any $x, y\in X$,
\newline
(3) $(\id_X \otimes \epsilon^0 )\circ \lambda =\id_X$,
\newline
(4) $(\lambda \otimes \id)\circ \lambda =(\id\otimes \Delta^0 )\circ\lambda$.
\newline
We call $(B, X, \sigma, \lambda, H^0 )$ a
\sl
right covariant system
\rm
(See \cite [Definition 3.4]{KT3:equivalence}).
Then we can construct an action `` $\cdot $" of $H$ on $\BK_B (X)$ as follows:
For any $a\in\BB_B (X)$, $h\in H$ and $x\in X$,
$$
[h\cdot a]x=h_{(1)}\cdot_{\lambda}a[S(h_{(2)})\cdot_{\lambda}x] .
$$
If $a\in \BK_B (X)$, we can see that $h\cdot a\in\BK_B (X)$. Thus identifying $A$ with $\BK_B (X)$,
we can obtain an action of $H$ on $A$.
\end{remark}

\section{Linking $C^*$-algebras and coactions on $C^*$-algebras}\label{sec:stable}

Let $(\rho, u)$ and $(\sigma, v)$ be twisted coactions of $H^0$ on $C^*$-algebras $A$ and $B$, respectively.
We suppose that there are a Hilbert $A-B$- bimodule $X$ and a twisted coaction $\lambda$ of $H^0$ on $X$ with
respect to $(A, B, \rho, u, \sigma , v)$. Let $C$ be the linking $C^*$-algebra for $X$ defined in
Brown, Mingo and Shen \cite {BMS:quasi}. By  \cite [Proposition 2.3]
{BMS:quasi}, $C$ is the $C^*$-algebra which is
consisting of all $2\times2$-matrices
$$
\begin{bmatrix}
a & x \\
\widetilde{y} & b
\end{bmatrix}, \quad a\in A, \quad b\in B,\quad x, y\in X,
$$
where $\widetilde{y}$ denotes $y$ viewed as an element in $\widetilde{X}$,
the dual Hilbert $C^*$-bimodule of $X$. Before we define the coaction of $H^0$
on $C$ induced by the twisted coaction $\lambda$ of $H^0$ on $X$ with
respect to $(A, B, \rho,u, \sigma, v)$, we give a remark.

\begin{remark}\label{remark:another} We identify the $H^0 -H^0$-equivalence bimodule $\widetilde{H^0}$ with $H^0$ as
$H^0 -H^0$-equivalence bimodule by the map
$$
\widetilde{H^0}\longrightarrow H^0 : \widetilde{\phi}\mapsto \phi^* .
$$
Also, we identify the Hilbert $B\otimes H^0 -A\otimes H^0$-bimodule $\widetilde{X\otimes H^0}$
with $\widetilde{X}\otimes H^0$ by the map
$$
\widetilde{X\otimes H^0}\longrightarrow \widetilde{X}\otimes H^0 : \widetilde{x\otimes\phi}\mapsto\widetilde{x}\otimes\phi^* .
$$
Furthermore, we identify the linking $C^*$-algebra for $X\otimes H^0$,
the Hilbert $A\otimes H^0 -B\otimes H^0$-bimodule with $C\otimes H^0$ by the isomorphism
defined by
\begin{align*}
& \Phi(\begin{bmatrix} a\otimes\phi_{11} & x\otimes\phi_{12} \\
\widetilde{y\otimes\phi_{21}} & b\otimes\phi_{22} \end{bmatrix}) \\
& =\begin{bmatrix} a & 0 \\
0 & 0 \end{bmatrix}\otimes\phi_{11}+
\begin{bmatrix} 0 & x \\
0 & 0 \end{bmatrix}\otimes\phi_{12}+
\begin{bmatrix} 0 & 0 \\
\widetilde{y} & 0 \end{bmatrix}\otimes\phi_{21}^* +
\begin{bmatrix} 0 & 0 \\
0 & b \end{bmatrix}\otimes\phi_{22} ,
\end{align*}
where $a\in A$, $b\in B$, $x, y\in X$ and $\phi_{ij}\in H^0$ $(i, j=1,2)$.
\end{remark}

Let $\gamma$ be the homomorphism of $C$ to $C\otimes H^0$ defined by
for any $a\in A$, $b\in B$, $x, y\in X$,
$$
\gamma(\begin{bmatrix} a & x \\
\widetilde{y} & b \end{bmatrix})=\begin{bmatrix} \rho(a) & \lambda(x) \\
\widetilde{\lambda(y)} & \sigma(b) \end{bmatrix} .
$$
Let $w$ be the unitary element in $M(C)$ defined by
$w=\begin{bmatrix} u & 0 \\
0 & v \end{bmatrix}$.
By routine computations, $(\gamma, w)$ is a twisted coaction of $H^0$ on $C$.

\begin{remark}\label{remark:linking} (1) We note that the twisted action of $H$ on $C$
induced by $(\gamma, w)$ as follows: For any $a\in A, b\in B, x,y\in X$ and $h\in H$,
$$
h\cdot_{\gamma}\begin{bmatrix} a & x \\
\widetilde{y} & b \end{bmatrix}
=\begin{bmatrix} h\cdot_{\rho, u}a & h\cdot_{\lambda}x \\
\widetilde{S(h)^* \cdot_{\lambda}y} & h\cdot_{\sigma, v}b \end{bmatrix} .
$$
(2) Let $\widetilde{\lambda}$ be a linear map from $X$ to $X\otimes H^0$ defined by for any
$x\in X$,
$$
\widetilde{\lambda}(\widetilde{x})=\widetilde{\lambda(x)} .
$$
Then
$\widetilde{\lambda}$ is the coaction of $H^0$ on $\widetilde{X}$ induced by $\lambda$.
Also, the twisted action of $H$ on $\widetilde{X}$ induced by $\widetilde{\lambda}$ is as follows:
For any $x\in X , h\in H$,
$$
h\cdot_{\widetilde{\lambda}}\widetilde{x}=\widetilde{S(h^* )\cdot_{\lambda}y} .
$$
\end{remark}

Let $C_1$ be the linking $C^*$-algebra for the Hilbert $A\rtimes_{\rho}H-B\rtimes_{\sigma}H$-
bimodule $X\rtimes_{\lambda}H$. Then we obtain the following lemma
by Remarks \ref{remark:another} and \ref{remark:linking}:

\begin{lemma}\label{lem:C1}With the above notations, there is an isomorphism
$\pi_1$ of $C\rtimes_{\gamma, w}H$ onto $C_1$.
\end{lemma}
\begin{proof}Let $\pi_1$ be the map from $C\rtimes_{\gamma, w}H$ to $C_1$ defined by
$$
\pi_1 (\begin{bmatrix} a & x \\
\widetilde{y} & b \end{bmatrix} \rtimes_{\gamma, w} h)
=\begin{bmatrix}a\rtimes _{\rho, u}h & x\rtimes_{\lambda}h \\
\{\widehat{u}(S(h_{(2)}), h_{(1)})^* [h_{(3)}^* \cdot_{\lambda}y]\rtimes_{\lambda}h_{(4)}^* \}{\widetilde{}} & b\rtimes_{\sigma, v}h
\end{bmatrix}
$$
for any $a\in A$, $b\in B$, $x, y\in X$ and $h\in H$. Let $\theta_1$ be the map
from $C_1$ to $C\rtimes_{\gamma, w}H$ defined by
\begin{align*}
& \theta_1 (\begin{bmatrix} a\rtimes_{\rho, u}h &  x\rtimes_{\lambda}l \\
\widetilde{y\rtimes_{\lambda}k} & b\rtimes_{\sigma, v}m \end{bmatrix}) \\
& =\begin{bmatrix} a & 0 \\
0 & 0 \end{bmatrix} \rtimes_{\gamma, w}h
+\begin{bmatrix} 0 & x \\
0 & 0 \end{bmatrix} \rtimes_{\gamma, w}l
+\begin{bmatrix} 0 & 0 \\
\{[S(h_{(3)}\cdot_{\lambda}y\widehat{v}(S(k_{(2)}), k_{(1)})\}^{\widetilde{}} & 0 \end{bmatrix} \rtimes_{\gamma, w}k_{(4)}^* \\
& +\begin{bmatrix} 0 & 0 \\
0 & b \end{bmatrix} \rtimes_{\gamma, w}m
\end{align*}
for any $a\in A$, $b\in B$, $x, y\in X$ and $h, k, l, m\in H$.
Then by routine computations, $\pi_1$ is a homomorphism of
$C\rtimes_{\gamma, w}H$ to $C_1$ and $\theta_1$ is a homomorphism of $C_1$ to $C\rtimes_{\gamma, w}H$
Moreover, we can see that $\theta_1$ is the inverse map
of $\pi_1$. Therefore, we obtain the conclusion.
\end{proof}

By the proof of the above lemma, we obtain the following corollary:

\begin{cor}\label{cor:product2}With the above notations, there is a Hilbert
$B\rtimes_{\sigma}H-A\rtimes_{\rho}H$-bimodule isomorphism $\pi$ of
$\widetilde{X\rtimes_{\lambda}H}$ onto $\widetilde{X}\rtimes_{\widetilde{\lambda}}H$.
\end{cor}

\begin{remark}\label{remark:conjugate2}Let $\gamma_1$ be a coaction of $H$ on $C_1$
defined by
$$
\gamma_1 =(\pi_1 \otimes \id_H )\circ\widehat{\gamma}\circ\pi_1^{-1} .
$$
Then by routine computations, for any $a\in A$, $b\in B$, $x, y\in X$ and
$h, l, k, m\in H$,
\begin{align*}
& \gamma_1 (\begin{bmatrix} a\rtimes_{\rho, u}h & x\rtimes_{\lambda}l \\
\widetilde{y\rtimes_{\lambda}k} & b\rtimes_{\sigma, v}m \end{bmatrix} ) \\
& =\begin{bmatrix} a\rtimes_{\rho, u}h_{(1)} & 0 \\
0 & 0 \end{bmatrix}\otimes h_{(2)}
+\begin{bmatrix} 0 & x\rtimes_{\lambda}l_{(1)} \\
0 & 0 \end{bmatrix} \otimes l_{(2)}
+\begin{bmatrix} 0 & 0 \\
(y\rtimes_{\lambda}k_{(1)})^{\widetilde{}} & 0 \end{bmatrix} \otimes k_{(2)}^* \\
& +\begin{bmatrix} 0 & 0 \\
0 & b\rtimes_{\sigma, v}m_{(1)} \end{bmatrix} \otimes m_{(2)} .
\end{align*}
\end{remark}

We give a result similar to \cite [Theorem 6.4]{KT1:inclusion} for coactions of $H^0$
on a Hilbert $C^*$-bimodule applying Proposition \ref{prop:nonunital2} to a linking $C^*$-algebra.
Let $\rho$ and $\sigma$ be coactions of  $H^0$ on $A$ and $B$, respectively and
let $X$ be a Hilbert $A-B$-bimodule. Let $\lambda$ be a coaction of $H^0$ on $X$
with respect to $(A, B, \rho, \sigma)$. Let $C$ be the linking $C^*$-algebra and
$\gamma$ the coaction of $H^0$ on $C$ induced by $\rho, \sigma$ and $\lambda$.
As defined in Section \ref{sec:morita}, let
$$
X^{\lambda}=\{x\in X \, | \, \lambda(x)=x\otimes 1^0 \}.
$$
Then by Lemma \ref{lem:full}, $X^{\lambda}$ is an Hilbert $A^{\rho}-B^{\sigma}$-
bimodule. Let $C_0$ be the linking $C^*$-algebra for $X^{\lambda}$.
We can prove the following lemma in the straightforward way. So, we give it without its proof.

\begin{lemma}\label{lem:fix3}With the above notations and assumptions, $C^{\gamma}=C_0$,
where $C^{\gamma}$ is the fixed point $C^*$-subalgebra of $C$ for $\gamma$.
\end{lemma}

\begin{lemma}\label{lem:sim1}With the above notations, if $\widehat{\underline{\rho}}(1\rtimes_{\underline{\rho}}e)
\sim(1\rtimes_{\underline{\rho}}e)\otimes 1$ in $(M(A)\rtimes_{\underline{\rho}}H)\otimes H$
and $\widehat{\underline{\sigma}}(1\rtimes_{\underline{\rho}}e)
\sim(1\rtimes_{\underline{\sigma}}e)\otimes 1$ in $(M(B)\rtimes_{\underline{\sigma}}H)\otimes H$,
then $\widehat{\underline{\gamma}}(1_{M(C)}\rtimes_{\underline{\gamma}}e)
\sim(1_{M(C)}\rtimes_{\underline{\gamma}}e)\otimes 1$ in
$(M(C)\rtimes_{\underline{\gamma}}H)\otimes H$.
\end{lemma}
\begin{proof}
By Remark \ref{lem:C1}, we identify $C\rtimes_{\gamma}H$ with
$C_1$, the linking $C^*$-algebra for the Hilbert $A\rtimes_{\rho}H-A\rtimes_{\rho}H$-
bimodule $X\rtimes_{\lambda}H$. Also, we identify $\widehat{\gamma}$ with
$\gamma_1$, the coaction of $H$ on $C_1$ defined in Remark \ref{remark:conjugate2}.
Hence
$$
\widehat{\underline{\gamma}}(1\rtimes_{\underline{\gamma}}e)=
\begin{bmatrix} 1\rtimes_{\underline{\rho}}e_{(1)} & 0 \\
0 & 0 \end{bmatrix}\otimes e_{(2)}
+\begin{bmatrix} 0 & 0 \\
0 & 1\rtimes_{\underline{\sigma}}e_{(1)} \end{bmatrix}\otimes e_{(2)} .
$$
By the assumptions,
\begin{align*}
& \begin{bmatrix} 1\rtimes_{\underline{\rho}}e_{(1)} & 0 \\
0 & 0 \end{bmatrix}\otimes e_{(2)}\sim\begin{bmatrix} 1\rtimes_{\underline{\rho}}e & 0 \\
0 & 0 \end{bmatrix}\otimes 1 \quad\text{in $\begin{bmatrix} M(A)\rtimes_{\underline{\rho}}H & 0 \\
0 & 0 \end{bmatrix} \otimes H$}, \\
& \begin{bmatrix} 0 & 0 \\
0 & 1\rtimes_{\underline{\sigma}}e_{(1)} \end{bmatrix}\otimes e_{(2)} \sim
\begin{bmatrix} 0 & 0 \\
0 & 1\rtimes_{\underline{\sigma}}e \end{bmatrix}\otimes 1 \quad\text{in $\begin{bmatrix} 0 & 0 \\
0 & M(A)\rtimes_{\underline{\sigma}}H \end{bmatrix} \otimes H$} .
\end{align*}
Since $\begin{bmatrix} M(A)\rtimes_{\underline{\rho}}H & 0 \\
0 & 0 \end{bmatrix}$ and $\begin{bmatrix} 0 & 0 \\
0 & M(A)\rtimes_{\underline{\sigma}}H \end{bmatrix}$ are $C^*$-subalgebras of $M(C_1 )$ by
the proof of Echterhoff and Raeburn \cite [Proposition A.1]{ER:multiplier},
$$
\begin{bmatrix} 1\rtimes_{\underline{\rho}}e_{(1)} & 0 \\
0 & 0 \end{bmatrix}\otimes e_{(2)}
+\begin{bmatrix} 0 & 0 \\
0 & 1\rtimes_{\underline{\sigma}}e_{(1)} \end{bmatrix}\otimes e_{(2)}\sim
\begin{bmatrix} 1\rtimes_{\underline{\rho}}e & 0 \\
0 & 1\rtimes_{\underline{\sigma}}e \end{bmatrix}\otimes 1
$$
in $M(C_1 )\otimes H$. Therefore, we obtain the conclusion since $M(C_1 )\otimes H$ is
identified with $(M(C)\rtimes_{\underline{\gamma}}H)\otimes H$.
\end{proof}

By \cite [Section 4]{KT1:inclusion}, there is a unitary element $w^{\rho}\in M(A)\otimes H$ satisfying that
\begin{align*}
& w^{\rho *}((1\rtimes_{\underline{\rho}}e)\otimes 1)w^{\rho}=\widehat{\underline{\rho}}(1\rtimes_{\underline{\rho}}e) , \\
& U^{\rho}=w^{\rho}(z^{\rho *}\otimes 1) , \quad
z^{\rho}=(\id_{M(A)}\otimes\epsilon)(w^{\rho})\in M(A)^{\underline{\rho}} \, .
\end{align*}
Also, there is a unitary element $w^{\sigma}\in M(B)\otimes H$ satisfying that
\begin{align*}
& w^{\sigma *}((1\rtimes_{\underline{\sigma}}e)\otimes 1)w^{\sigma}=\widehat{\underline{\sigma}}
(1\rtimes_{\underline{\sigma}}e) , \\
& U^{\sigma}=w^{\sigma}(z^{\sigma *}\otimes 1) , \quad
z^{\sigma}=(\id_{M(A)}\otimes\epsilon)(w^{\sigma})\in M(A)^{\underline{\sigma}} \, .
\end{align*}

Let $w^{\gamma}=\begin{bmatrix} w^{\rho} & 0 \\
0 & w^{\sigma} \end{bmatrix}\in M(C)\otimes H$. Then $w^{\gamma}$ is a unitary element
satisfying that $w^{\gamma*}((1\rtimes_{\underline{\gamma}}e)\otimes )w^{\gamma}
=\widehat{\underline{\gamma}}(1\rtimes_{\underline{\gamma}}e)$.
Let $U^{\gamma}=w^{\gamma}(z^{\gamma*}\otimes 1)$, where $z^{\gamma}=(\id_{M(C)}\otimes\epsilon)(w^{\gamma})
\in M(C)^{\underline{\gamma}}$. Then by Section \ref {sec:pre}, $U^{\gamma}$ satisfies that
$$
\widehat{U}^{\gamma}(1^0 )=1, \quad \widehat{U}^{\gamma}(\phi_{(1)})c\widehat{U}^{\gamma *}(\phi_{(2)})\in M(C)^{\underline{\gamma}}
$$
for any $c\in M(C)^{\underline{\gamma}}$, $\phi\in H^0$. Let $(\eta, u^{\gamma})$ be a twisted
coaction of $H$ on $C^{\gamma}$ induced by $U^{\gamma}$ which is defined in Section \ref{sec:pre}.
Then by the proof of Proposition \ref {prop:nonunital2}, there is the isomorphism $\pi_C$ of
$C^{\gamma}\rtimes_{\eta, u^{\gamma}}H^0$ onto $C$ defined by
$$
\pi_C (c\rtimes_{\eta, u^{\gamma}}\phi)=c\widehat{U}^{\gamma}(\phi)
$$
for any $c\in C^{\gamma}, \phi\in H^0$, which satisfies that
$$
\gamma\circ\pi_C =(\pi_C\otimes\id_H )\circ\widehat{\eta}, \quad E^{\eta. u^{\gamma}}=E^{\gamma}\circ\pi_C ,
$$
where $E^{\eta, u^{\gamma}}$ and $E^{\gamma}$ are the canonical conditional expectations from
$C^{\gamma}\rtimes_{\eta, u^{\gamma}}H^0$ and $C$ onto $C^{\gamma}$, respectively.
Let $p=\begin{bmatrix} 1_A & 0 \\
0 & 0 \end{bmatrix}$, $q=\begin{bmatrix} 0 & 0 \\
0 & 1_B \end{bmatrix}$.
Then $p$ and $q$ are projection in $M(C^{\gamma})$. We note that $M(C^{\gamma})=M(C)^{\underline{\gamma}}$ by
Lemma \ref{lem:fix}.

\begin{lemma}\label{lem:restriction}With the above notations and assumptions,
\begin{align*}
& \pi_C (p\rtimes_{\underline{\eta}, u^{\gamma}}1^0 )=p, \quad
u^{\gamma}(p\otimes 1\otimes 1)=(p\otimes 1\otimes 1)u^{\gamma}, \\
& \pi_C (q\rtimes_{\underline{\eta}, u^{\gamma}}1^0 )=q, \quad
u^{\gamma}(q\otimes 1\otimes 1)=(q\otimes 1\otimes 1)u^{\gamma},
\end{align*}
\end{lemma}
\begin{proof}
We note that $C^{\gamma}$ is identified with the $C^*$-subalgebra
$C^{\gamma}\rtimes_{\eta, u^{\gamma}}1^0$ of $C^{\gamma}\rtimes_{\eta, u^{\gamma}}H^0$.
Then by \cite [Proposition 2.12]{SP:saturated},
\begin{align*}
p & =E_1^{\eta, u^{\gamma}}(p\rtimes_{\underline{\eta}, u^{\gamma}}1^0 )
=E^{\gamma}(\pi_C (p\rtimes_{\underline{\eta}, u^{\gamma}}1^0 ))
=e\cdot_{\gamma}\pi_C (p\rtimes_{\underline{\eta}, u^{\gamma}}1^0 ) \\
& =\pi_C (e\cdot_{\widehat{\eta}}(p\rtimes_{\underline{\eta}, u^{\gamma}}1^0 ))=\pi_C (p)
=\pi_C (p\rtimes_{\underline{\eta}, u^{\gamma}}1^0 )
\end{align*}
since $\gamma\circ\pi_C =(\pi_C \otimes\id_H )\circ\widehat{\eta}$. Similarly,
we can obtain that $\pi_C (q\rtimes_{\underline{\eta}, u^{\gamma}}1^0 )=q$.
Furthermore, by the deifnition of $U^{\gamma}$, $U^{\gamma}=
\begin{bmatrix} U^{\rho} & 0 \\
0 & U^{\sigma} \end{bmatrix}\in M(C)\otimes H$. Hence
$U^{\gamma}(p\otimes 1)=(p\otimes 1)U^{\gamma}$. Since
$$
\widehat{u}^{\gamma}(\phi, \psi)
=\widehat{U}^{\gamma}(\phi_{(1)})\widehat{U}^{\gamma}(\psi_{(1)})\widehat{U}^{\gamma *}(\phi_{(2)}\psi_{(2)})
$$
for any $\phi, \psi\in H^0 $, we can see that $u^{\gamma}(p\otimes 1\otimes 1)=(p\otimes 1\otimes 1)u^{\gamma}$.
Similarly $u^{\gamma}(q\otimes 1\otimes 1)=(q\otimes 1\otimes 1)u^{\gamma}$.
\end{proof}

Let $\alpha=\eta|_{A^{\rho}}$, $\beta=\eta|_{B^{\sigma}}$ and $\mu=\eta|_{X^{\lambda}}$.
Let $u^{\rho}=u^{\gamma}(p\otimes 1\otimes 1)$, $u^{\sigma}=u^{\gamma}(q\otimes 1\otimes 1)$
Furthermore, let $\pi_A =\pi_C |_A$, $\pi_B =\pi_C |_B$, $\pi_X =\pi_C |_X$.
Then $(\alpha, u^{\rho})$ and $(\beta, u^{\sigma})$ are twisted coactions o $H^0$ on $A^{\rho}$ and $B^{\sigma}$,
respectively and $\mu$ is a twisted coaction of $H^0$ on $X^{\lambda}$
with respect to $(A, B, \alpha, u^{\rho}, \beta, u^{\sigma})$. Also, $\pi_A$ and $\pi_B$ are isomorphisms
of $A^{\rho}\rtimes_{\alpha, u^{\rho}}H^0$ and $B^{\sigma}\rtimes_{\beta, u^{\sigma}}H^0$ onto
$A$ and $B$ satisfying the results in Proposition \ref {prop:nonunital2}, respectively.
Furthermore, we obtain the following:

\begin{thm}\label{thm:surjective}Let $A$ and $B$ be $C^*$-algebras and $H$ a finite
dimensiona $C^*$-Hopf algebra with its dual $C^*$-Hopf algebra $H^0$.
Let $\rho$ and $\sigma$ be coactions of $H^0$
on $A$ and $B$, respectively. Let $\lambda$ be a coaction of $H^0$ on
a Hilbert $A-B$-bimodule $X$ with respect to $(A, B, \rho, \sigma)$.
We suppose that $\widehat{\underline{\rho}}(1\rtimes_{\underline{\rho}}e)
\sim (1\rtimes_{\underline{\rho}}e)\otimes 1$ in $M(A)\rtimes_{\underline{\rho}}H$
and that $\widehat{\underline{\sigma}}(1\rtimes_{\underline{\sigma}}e)\sim(1\rtimes_{\underline{\sigma}}e)
\otimes 1$ in $M(B)\rtimes_{\underline{\sigma}}H$. Then there are a twisted coaction $\mu$
of $H^0$ on $X^{\lambda}$ and a bijective linear map $\pi_X$ from $X^{\lambda}\rtimes_{\mu}H^0$
onto $X$ satisfying the following conditions:
\newline
$(1)$ $\pi_X ((a\rtimes_{\alpha, u^{\rho}}\phi)(x\rtimes_{\mu}\psi))
=\pi_A (a\rtimes_{\alpha, u^{\rho}}\phi)\pi_X (x\rtimes_{\mu}\psi)$,
\newline
$(2)$ $\pi_X((x\rtimes_{\mu}\phi)(b\rtimes_{\beta, u^{\sigma}}\psi))
=\pi_X (x\rtimes_{\mu}\phi)\pi_B (b\rtimes_{\beta, u^{\sigma}}\psi)$,
\newline
$(3)$ $\pi_A ({}_{A^{\rho}\rtimes_{\alpha, u^{\rho}}H^0} \la x\rtimes_{\mu}\phi, \, y\rtimes_{\mu}\psi \ra)
={}_A \la \pi_X (x\rtimes_{\mu}\phi), \, \pi_X (y\rtimes_{\mu}\psi) \ra$,
\newline
$(4)$ $\pi_B (\la x\rtimes_{\mu}\phi, \, y\rtimes_{\mu}\psi \ra_{B^{\sigma}\rtimes_{\beta, u^{\sigma}}H^0})
=\la \pi_X (x\rtimes_{\mu}\phi), \, \pi_X(y\rtimes_{\mu}\psi) \ra_B$,
\newline
$(5)$ $h\cdot_{\lambda}\pi_X (x\rtimes_{\mu}\phi)=\pi_X (h\cdot_{\widehat{\mu}}(x\rtimes_{\mu}\phi))$,
for any $x,y\in X^{\lambda}$, $a\in A^{\rho}$, $b\in B^{\sigma}$, $h\in H$, $\phi, \psi\in H^0$.
\end{thm}
\begin{proof}Using the above discussions, we can prove the theorem in a straightforward way.
\end{proof}

Let $A$ be a unital $C^*$-algebra and $\rho$ a coaction of $H^0$ on $A$.
Let $\BK$ be the $C^*$-algebra of all compact operators on a countably
infinite dimensional Hilbert space. Let $A^s =A\otimes\BK$ and $\rho^s =\rho\otimes\id$.
We identify $H^0 \otimes\BK$ with $\BK\otimes H^0$. Then $\rho^s$ is a coaction
of $H^0$ on $A^s$.

\begin{lemma}\label{lem:stable}With the above notations, $\rho$ and $\rho^s$ are strongly
Morita equivalent.
\end{lemma}
\begin{proof}This is immediate by routine computations.
\end{proof}

Let $A$ and $B$ be unital $C^*$-algebras. Let $\rho$ and $\sigma$ be coactions
of $H^0$ on $A$ and $B$, respectively. We suppose that $\rho$ and $\sigma$ are
strongly Morita equivalent. Also, we suppose that there are an $A-B$-equivalence
bimodule $X$ and a coaction $\lambda$ of $H^0$ on $X$ with respect to $(A, B, \rho, \sigma)$.
Let $C$ be the linking $C^*$-algebra for $X$ and $\gamma$ the coaction of $H^0$ on $C$
induced by $\rho, \sigma$ and $\lambda$, which is defined in the above.
Let $A^s =A\otimes\BK$, $B^s =B\otimes\BK$ and $C^s =C\otimes\BK$. Let $X^s =X\otimes\BK$,
the exterior tensor product of $X$ and $\BK$, which is an $A^s -B^s$-equivalence bimodule in
the usual way. Let $\rho^s =\rho\otimes\id$, $\sigma^s =\sigma\otimes\id$ and
$\gamma^s =\gamma\otimes\id$. Let $\lambda^s =\lambda\otimes\id$, which is a coaction of $H^0$
on $X^s$. Let
$$
p=\begin{bmatrix} 1_A \otimes 1_{M(\BK)} & 0 \\
0 & 0 \end{bmatrix} , \quad
q=\begin{bmatrix} 0 & 0 \\
0 & 1_B \otimes 1_{M(\BK)} \end{bmatrix} .
$$
Then $p$ and $q$ are full projections in $M(C^s )$ and
$A^s \cong pC^s p$, $B^s \cong qC^s q$. We identify $A^s$ and $B^s$ with
$pC^s p$ and $qC^s q$, respectively. By Brown \cite [Lemma 2.5]{Brown:hereditary},
there is a partial isometry $w\in M(C^s )$ such that $w^*w=p$, $ww^* =q$.
Let $\theta$ be a map from $A^s$ to $C^s$ defined by
$$
\theta(a)=waw^* =w\begin{bmatrix} a & 0 \\
0 & 0 \end{bmatrix}w^*
$$
for any $a\in A$. Since $w^* w=p$ and $ww^* =q$, by easy computations, we can see that
$\theta$ is an isomorphism of $A^s$ onto $B^s$.

\begin{prop}\label{prop:stable}With he above notations, there is a unitary element
$u\in M(B^s )\otimes H^0$ such that
\begin{align*}
(\theta\otimes\id_{H^0 })\circ\rho^s \circ \theta^{-1} & =\Ad(u)\circ\sigma^s , \\
(u\otimes1^0 )(\underline{\sigma^s} \otimes\id_{H^0 })(u) & =(\id_{M(B^s )}\otimes\Delta^0 )(u) ,
\end{align*}
where $\underline{\sigma^s}$ is the strictly continuous coaction of $H^0$ on $M(B^s )$ extending the
coaction $\sigma^s$ of $H^0$ on $B^s$.
\end{prop}
\begin{proof}
We note that $\theta=\Ad(w)$. Since $\rho^s =\gamma^s |_{A^s}$ and
$\sigma^s =\gamma^s |_{B^s}$, we can obtain that
$$
(\theta\otimes\id_{H^0})\circ \rho^s \circ \theta^{-1}=\Ad((w\otimes 1^0 )\underline{\gamma^s}(w^* ))\circ\sigma^s ,
$$
where $\underline{\gamma^s}$ is the strictly continuous coaction of $H^0$ on $M(C^s )$ extending the
coaction $\gamma^s$ of $H^0$ on $C^s$. Let $u=(w\otimes 1^0 )\underline{\gamma^s}(w^* )$.
By routine computations, we can show that $u$ is a desired unitary element in $M(B^s )\otimes H^0$.
\end{proof}

\section{Equivariant Picard groups}\label{sec:Picard}
Following Jansen and Waldmann \cite {JW:covariant}, we shall define the equivariant
Picard group of a $C^*$-algebra .
\par
Let $A$ be a $C^*$-algebra and $H$ a finite dimensiona $C^*$-Hopf algebra
with its dual $C^*$-Hopf algebra $H^0$. Let $(\rho, u)$ be a twisted coaction of $H^0$ on $A$.
We denote by $(X, \lambda)$, a pair of an $A-A$-equivalence bimodule $X$ and a twisted
coaction $\lambda$ of $H^0$ on $X$ with respect to $(A, A, \rho, u, \rho, u)$.
Let $\Equi_H ^{\rho, u}(A)$ be the set of all such pairs $(X, \lambda)$ as above. We define
an equivalence relation $\sim$ in $\Equi_H^{\rho, u}(A)$ as follows: For $(X, \lambda), (Y, \mu)
\in \Equi_H^{\rho, u}(A)$, $(X, \lambda)\sim (Y, \mu)$ if and only if there is an $A-A$-equivalence
bimodule isomorphism $\pi$ of $X$ onto $Y$ such that $\mu\circ\pi=(\pi\otimes\id_{H^0})\circ\lambda$,
that is, for any $x\in X$ and $h\in H$, $\pi(h\cdot_{\lambda}x )=h\cdot_{\mu}\pi (x)$.
We denote by $[X, \lambda]$ the equivalence class of $(X, \lambda)$ in $\Equi_H^{\rho, u}(A)$.
Let $\Pic_H^{\rho, u}(A)=\Equi_H^{\rho, u}(A)/\!\sim$. We define the product in $\Pic_H^{\rho, u}(A)$
as follows: For $(X, \lambda), (Y, \mu)\in \Equi_H^{\rho, u}(A)$,
$$
[X, \lambda][Y, \mu]=[X\otimes_A Y, \, \lambda\otimes\mu],
$$
where $\lambda\otimes\mu$ is the twisted coaction of $H^0$ on $X$ induced by the 
action $ `` \cdot_{\lambda\otimes\mu} "$ of $H$ on $X$ defined in \cite [Proposition 3.1]
{KT3:equivalence}. By easy computations, we can see that the above product is well-defined.
We regard $A$ as an $A-A$-equivalence bimodule in the usual way.
We sometimes denote it by ${}_A A_A$. Also, we can regard a twisted coaction $\rho$ of $H^0$ on
$C^*$-algebra $A$ as a twisted coaction of $H^0$ on the $A-A$-equivalence bimodule ${}_A A_A$
with respect to $(A, A, \rho, u, \rho, u)$.
Then $[{}_A A_A, \, \rho]$ is the unit element in $\Pic_H^{\rho ,u}(A)$. Let $\widetilde{\lambda}$
be the coaction of $H^0$ on $\widetilde{X}$ defined by $\widetilde{\lambda}(\widetilde{x})=\widetilde{\lambda(x)}$
for any $x\in X$, which is also defined in Remark \ref{remark:linking} (2).
Then we can see that $[\widetilde{X}, \, \widetilde{\lambda}]$ is the inverse
element of $[X, \lambda]$ in $\Pic_H^{\rho, u}(A)$. By the above product,
$\Pic_H^{\rho, u}(A)$ is a group. We call it the
\sl
$(\rho, u, H)$-equivariant Picard group
\rm
of $A$.
\par
Let $\Aut_H^{\rho, u}(A)$ be the group of all automorphisms $\alpha$ of $A$ satisfying
that $(\alpha\otimes\id_{H^0})\circ\rho=\rho\circ\alpha$ and $\Int_H^{\rho, u}(A)$ the set of
all generalized inner automorphisms $\Ad(v)$ of $A$ satisfying that
$\underline{\rho}(v)=v\otimes 1^0$, where $v$ is a unitary element in $M(A)$.
By easy computations $\Int_H^{\rho, u}(A)$ is a normal subgroup of $\Aut_H^{\rho, u}(A)$.
Modifying \cite {BGR:linking}, for each $\alpha\in \Aut_H^{\rho, u}(A)$, we construct
the element $(X_{\alpha}, \lambda_{\alpha})\in\Equi_H^{\rho, u}(A)$ as follows: Let $\alpha\in\Aut_H^{\rho, u}(A)$.
Let $X_{\alpha}$ be the vector space $A$ with the obvious left action of $A$ on $X_{\alpha}$
and the obvious left $A$-valued inner product, but define the right action of $A$ on $X_{\alpha}$
by $x\cdot a=x\alpha(a)$ for any $x\in X_{\alpha}$, $a\in A$ and the right $A$-valued inner
product by $\la x, \, y \ra_A =\alpha^{-1}(x^* y)$ for any $x, y\in X_{\alpha}$. Then
by \cite {BGR:linking}, $X_{\alpha}$ is an $A-A$-equivalence bimodule. Also, $\rho$ can be regarded
as a linear map from $X_{\alpha}$ to an $A\otimes H^0 -A\otimes H^0$-equivalence
bimodule $X_{\alpha}\otimes H^0$. We denote it by $\lambda_{\alpha}$. By
easy computations, $\lambda_{\alpha}$ is a twisted coaction of $H^0$ on $X_{\alpha}$
with respect to $(A, A, \rho, u, \rho, u)$.
Thus we obtain the map
$\Phi$
$$
\Phi : \Aut_H ^{\rho, u}(A)\to \Pic_H^{\rho, u}(A): \alpha\mapsto [X_{\alpha}, \lambda_{\alpha}] .
$$
Modifying \cite {BGR:linking}, we can see that the map $\Phi$ is a homomorphism of $\Aut_H^{\rho, u}(A)$
to $\Pic_H^{\rho, u}(A)$. We have the similar result to \cite [Proposition 3.1]{BGR:linking}.

\begin{prop}\label{prop:exact}With the above notations, we have the
exact sequence
$$
1\to\Int_H^{\rho, u}(A)\overset{\imath}\to\Aut_H^{\rho, u}(A)\overset{\Phi}\to\Pic_H^{\rho, u}(A) ,
$$
where $\imath$ is the inclusion map of $\Int_H^{\rho, u}(A)$ to $\Aut_H^{\rho, u}(A)$.
\end{prop}
\begin{proof}
Modifying the proof of \cite [Proposition 3.1]{BGR:linking}, we shall prove this Proposition.
Let $v$ be a unitary element in $M(A)$ with $\underline{\rho}(v)=v\otimes 1^0$.
We show that $[X_{\Ad(v)}, \, \lambda_{\Ad(v)}]=[{}_A A_A, \, \rho]$ in $\Pic_H^{\rho, u}(A)$.
Let $\pi$ be the map from ${}_A A_A$ to $X_{\Ad(v)}$ defined by $\pi(a)=av^*$
for any $a\in {}_A A_A$. Then $\pi$ is an $A-A$-equivalence bimodule isomorphism.
Also, for any $a\in {}_A A_A$ and $h\in H$,
$$
h\cdot_{\lambda_{\Ad(v)}}\pi(a)=h\cdot_{\lambda_{\Ad(v)}}(av^* )=[h_{(1)}\cdot_{\rho}a][h_{(2)}\cdot_{\underline{\rho}}v^* ]
=[h\cdot_{\rho}a]v^* =\pi(h\cdot_{\rho}a).
$$
Thus $[X_{\Ad(v)}, \, \lambda_{\Ad(v)}]=[{}_A A_A , \, \rho]$ in $\Pic_H^{\rho, u}(A)$.
Conversely, let $\alpha\in \Aut_H^{\rho, u}(A)$ with
$[X_{\alpha}, \lambda_{\alpha}]=[{}_A A_A , \rho]$ in $\Pic_H^{\rho, u}(A)$. Then there is an $A-A$-equivalence bimodule
isomorphism $\pi$ of ${}_A A_A$ onto $X_{\alpha}$ such that
$$
\lambda_{\alpha}\circ\pi=(\pi\otimes\id)\circ\rho .
$$
By the proof of \cite [Proposition 3.1]{BGR:linking}, $(\pi\circ\alpha^{-1}, \, \pi)$ is
a double centralizer of $A$. Hence $(\pi\circ\alpha^{-1}, \, \pi)\in M(A)$.
Let $v=(\pi\circ\alpha^{-1}, \, \pi)$. Then $v$ is a unitary element in $M(A)$ such that
$\alpha=\Ad(v^* )$. Furthermore, since $\lambda_{\alpha}\circ\pi=(\pi\otimes\id)\circ\rho$,
for any $a\in A$, $\lambda_{\alpha}(\pi(a))=(\pi\otimes\id)(\rho(a))$.
It follows that $\rho(av^* )=\rho(a)(v\otimes 1^0 )^*$ for any $a\in A$.
That is, $\underline{\rho}(v)=v\otimes 1^0 $. Therefore, we obtain the conclusion.
\end{proof}

Next, we shall show a similar result to \cite [Corollary 3.5]{BGR:linking}. Let $A$ be
a $C^*$-algebra and $X$ an $A-A$-equivalence bimodule. Let $\rho$ be a coaction of $H^0$ on $A$
and $\lambda$ a coaction of $H^0$ on $X$ with respect to $(A, A, \rho, \rho)$.
Let $C$ be the linking $C^*$-algebra for $X$ and $\gamma$ the coaction of $H^0$
on $C$ induced by $\rho$ and $\lambda$ which is defined in Section \ref{sec:stable}.
Furthermore, we suppose that $A$ is unital and that $\widehat{\rho}(1\rtimes_{\rho}e)
\sim(1\rtimes_{\rho}e)\otimes 1$ in $(A\rtimes_{\rho}H)\otimes H$. Then $\rho$ is
saturated by \cite [Section 4]{KT1:inclusion}. Let $(\widehat{\rho})^s$ be the coaction of $H$
on $(A\rtimes_{\rho}H)^s \otimes H$ induced by the dual coaction $\widehat{\rho}$ of
$H$ on $A\rtimes_{\rho}H$. Also, let $(\rho^s )^{\widehat{}}$ be the dual coaction of $\rho^s$
which is a coaction of $H$ on $A^s \rtimes_{\rho^s}H$. By their definitions, we can see that
$(\widehat{\rho})^s =(\rho^s )^{\widehat{}}$, where we identify $(A\rtimes_{\rho}H)^s$ with
$A^s \rtimes_{\rho^s}H$. We denote them by $\widehat{\rho^s}$.

\begin{lemma}\label{lem:sim2}With the above notations, if $\widehat{\rho}(1\rtimes_{\rho}e)\sim
(1\rtimes_{\rho}e)\otimes 1$ in $(A\rtimes_{\rho}H)\otimes H$,
then $\widehat{\underline{\rho}^s} (1\rtimes_{\underline{\rho^s}}e)\sim (1\rtimes_{\underline{\rho^s}}e)
\otimes 1$ in $(M(A^s )\rtimes_{\underline{\rho^s}}H)\otimes H$.
\end{lemma}
\begin{proof}This is immediate by straightforward computations.
\end{proof}

Let $C$ be the linking $C^*$-algebra for an $A^s -A^s$-equivalence bimodule
$X^s$ and $\gamma$ the coaction of $H$ on $C$ induced by $\rho^s$ and $\lambda^s$.

\begin{lemma}\label{lem:sim3}With the above notations, if $\widehat{\rho}(1\rtimes_{\rho}e)
\sim(1\rtimes_{\rho}e)\otimes 1$ in $(A\rtimes_{\rho}H)\otimes H$, then
$\widehat{\gamma}(1_{M(C)}\rtimes_{\underline{\gamma}}e)\sim(1_{M(C)}\rtimes_{\underline{\gamma}}e)\otimes 1$
in $(M(C)\rtimes_{\underline{\gamma}}H)\otimes H$.
\end{lemma}
\begin{proof}This is immediate by Lemmas \ref{lem:sim1} and \ref{lem:sim2}.
\end{proof}

\begin{lemma}\label{lem:surjection2}With the above notations, we suppose that
$\widehat{\rho}(1\rtimes_{\rho}e)\sim (1\rtimes_{\rho}e)\otimes 1$ in $(A\rtimes_{\rho}H)\otimes H$.
Let $\Phi$ be the homomorphism of $\Aut_H^{\rho^s}(A^s )$ to $\Pic_H^{\rho^s}(A^s )$
defined by $\Phi(\alpha)=[X_{\alpha}, \, \lambda_{\alpha}]$
for any $\alpha\in\Aut_H^{\rho^s}(A^s )$. Then $\Phi$ is surjective.
\end{lemma}
\begin{proof}Let $[X, \lambda]$ be any element in $\Pic_H^{\rho^s}(A^s )$. Let
$$
X^{\lambda}=\{x\in X \, | \, \lambda(x)=x\otimes 1^0 \}.
$$
Since $\widehat{\rho}(1\rtimes_{\rho}e)\sim (1\rtimes_{\rho}e)\otimes 1$ in
$(A\rtimes_{\rho}H)\otimes H$, by Lemma \ref{lem:sim2},
$\widehat{\underline{\rho^s}}(1\rtimes_{\underline{\rho^s}}e)\sim
(1\rtimes_{\underline{\rho^s}}e)\otimes 1$ in $(M(A^s )\rtimes_{\underline{\rho^s}}H)\otimes H$.
Since $X$ is an $A^s -A^s$-equivalence bimodule, by Lemma \ref{lem:full} and
Theorem \ref{thm:surjective}, $X^{\lambda}$ is an $(A^s )^{\rho^s}-(A^s )^{\rho^s}$-
equivalence bimodule, where $(A^s )^{\rho^s}$ is the fixed point $C^*$-subalgebra of
$A^s$ for the coaction $\rho^s$. Let $C$ be the linking $C^*$-algebra for $X$ and $\gamma$
the coaction of $H^0$ on $C$ induced by $\rho^s$ and $\lambda$.
Let $C^{\gamma}$ be the fixed point $C^*$-algebra of $C$ for $\gamma$.
Then by Lemma \ref{lem:fix3}, $C^{\gamma}$ is isomorphic to $C_0$, the linking
$C^*$-algebra for $X^{\lambda}$. We identify $C^{\gamma}$ with $C_0$. Let
$$
p=\begin{bmatrix} 1_A \otimes 1_{M(\BK)} & 0 \\
0 & 0 \end{bmatrix}, \quad
q=\begin{bmatrix} 0 & 0 \\
0 & 1_A \otimes 1_{M(\BK)} \end{bmatrix}
$$
Then $p$ and $q$ are projections in $M(C)^{\underline{\gamma}}$.
Since $M(C)^{\underline{\gamma}}=M(C^{\gamma})$ by Lemmas \ref{lem:fix} and \ref{lem:sim1},
$p$ and $q$ are full for $C^{\gamma}$.
By the proof of \cite [Theorem 3.4]{BGR:linking}, there is a partial isometry
$w\in M(C)^{\underline{\gamma}}$ such that
$$
w^* w=p, \quad q=ww^* .
$$
Hence $w\in M(C)$.
Let $\alpha$ be the map on $A^s$ defined by
$$
\alpha(a)=w^* aw=w^* \begin{bmatrix} 0 & 0 \\
0 & a \end{bmatrix}w
$$
for any $a\in A^s$. By routine computations, $\alpha$ is an automorphism of $A^s$.
Let $\pi$ be a linear map from $X$ to $X_{\alpha}$ defined by
$$
\pi(x)=\begin{bmatrix} 0 & x \\
0 & 0 \end{bmatrix}w=p\begin{bmatrix} 0 & x \\
0 & 0 \end{bmatrix}wp
$$
for any $x\in X$. In the same way as in the proof of \cite [Lemma 3.3]{BGR:linking},
we can see that $\pi$ is an $A^s -A^s$-equivalence bimodule isomorphism of $X$
onto $X_{\alpha}$. For any $a\in A^s$,
\begin{align*}
(\rho^s \circ\alpha)(a) & =\rho^s (w^* aw)=\gamma(w^* \begin{bmatrix} 0 & 0 \\
0 &a \end{bmatrix}w)=\underline{\gamma}(w^* )\begin{bmatrix} 0 & 0 \\
0 & \rho^s (a) \end{bmatrix}\underline{\gamma}(w) \\
& =(\alpha\otimes\id_{H^0})(\rho^s (a))
\end{align*}
since $w\in M(C)^{\underline{\gamma}}$.
Hence $\alpha\in \Aut_H^{\rho^s}(A^s )$. Furthermore, for any $x\in X$,
\begin{align*}
(\lambda_{\alpha}\circ\pi)(x)& =\lambda_{\alpha}(\begin{bmatrix} 0 & x \\
0 & 0 \end{bmatrix}w) =\rho^s (\begin{bmatrix} 0 & x \\
0 & 0 \end{bmatrix}w) \\
& =\gamma(\begin{bmatrix} 0 & x \\
0 & 0 \end{bmatrix}w)=\begin{bmatrix} 0 & \lambda(x) \\
0 & 0 \end{bmatrix}(w\otimes 1^0 )
=(\pi\otimes\id_{H^0 })(\lambda(x)) ,
\end{align*}
where we identify $\BK\otimes H^0$ with $H^0 \otimes\BK$. Thus
$\Phi(\alpha)=[X, \lambda]$. Therefore, we obtain the conclusion.
\end{proof}

\begin{thm}\label{thm:exact2}Let $A$ be a unital $C^*$-algebra and $\rho$ a coaction of $H^0$ on $A$.
We suppose that $\widehat{\rho}(1\rtimes_{\rho}e)\sim(1\rtimes_{\rho}e)\otimes 1$ in
$(A\rtimes_{\rho}H)\otimes H$. Then we have the following exact sequence:
$$
1\to \Int_H^{\rho^s}(A^s )\overset{\imath}\to\Aut_H^{\rho^s}(A^s )\overset{\Phi}\to\Pic_H^{\rho^s}(A^s )\to 1 ,
$$
where $\imath$ is the inclusion map of $\Int_H^{\rho^s}(A^s )$ to $\Aut_H^{\rho^s}(A^s )$.
\end{thm}
\begin{proof}This is immediate by Proposition \ref{prop:exact} and Lemma \ref{lem:surjection2}.
\end{proof}

Since the following lemma is obtained in a straightforward way, we omit its proof:

\begin{lemma}\label{lem:isom}Let $(\rho, u)$ and $(\sigma, v)$ be twisted
coactions on $C^*$-algebras $A$ and $B$, respectively. We suppose that $(\rho, u)$
is strongly Morita equivalent to $(\sigma, v)$. Then $\Pic_H^{\rho, u}(A)\cong\Pic_H^{\sigma, v}(B)$.
\end{lemma}

\section{Ordinary Picard groups and equivariant Picard groups}\label{sec:ordinary}
In this section, we shall investigate the relation between ordinary Picard groups and
equivariant Picard groups.
Let $\rho$ be a coaction of $H^0$ on a $C^*$-algebra $A$ and
let $f_{\rho}$ be the map from $\Pic_H^{\rho}(A)$ to $\Pic(A)$ defined by
$$
f_{\rho}:\Pic_H^{\rho}(A)\to\Pic(A): [X, \lambda]\mapsto [X],
$$
where $\Pic(A)$ is the ordinary Picard group of $A$. Clearly $f_{\rho}$ is a homomorphism
of $\Pic_H^{\rho}(A)$ to $\Pic(A)$. Let $\Aut(A)$ be the group of all automorphisms of $A$ and
let $\alpha\in\Aut(A)$. Let $X_{\alpha}$ be the $A-A$-equivalence bimodule induced by $\alpha$
defined in Section \ref{sec:Picard}.
Let $\lambda$ be a coaction of $H^0$ on $X_{\alpha}$ with respect to $(A, A, \rho, \rho)$.
Then for any $a\in A$ and $x, y\in X_{\alpha}$,
\newline
(1) $\lambda(ax)=\lambda(a\cdot x)=\rho(a)\cdot \lambda(x)=\rho(a)\lambda(x)$,
\newline
(2) $\lambda(x\alpha(a))=\lambda(x\cdot a)=\lambda(x)\cdot \rho(a)=\lambda(x)(\alpha\otimes\id)(\rho(a))$,
\newline
(3) $\rho(xy^* )=\rho({}_A \la x, y \ra)={}_{A\otimes H^0 } \la \lambda(x), \lambda(y) \ra=\lambda(x)\lambda(y)^*$,
\newline
(4) $\rho(\alpha^{-1}(x^* y))=\rho(\la x, y \ra_A )=\la \lambda(x), \lambda(y) \ra_{A\otimes H^0}
=(\alpha^{-1}\otimes\id)(\lambda(x)^* \lambda(y))$,
\newline
(5) $(\id\otimes\epsilon^0 )(\lambda(x))=x$,
\newline
(6) $(\lambda\otimes\id)(\lambda(x))=(\id\otimes\Delta^0 )(\lambda(x))$.
\newline
Let $\{u_{\gamma}\}$ be an approximate unit of $A$. Then $\lambda(u_{\gamma})\in X_{\alpha}\otimes H^0$.
Since $X_{\alpha}=A$ as vector spaces, we regard $\lambda(u_{\gamma})$ as an element in $A\otimes H^0$.

\begin{lemma}\label{lem:unitary}With the above notations, we regard $\lambda(u_{\gamma})$ as an
element in $A\otimes H^0$. Then $\{\lambda(u_{\gamma})\}$ converges to a unitary element in
$M(A\otimes H^0 )$ strictly  and the unitary element does not depend on the choice of an approximate unit of $A$.
\end{lemma}
\begin{proof}Let $a\in A$ and $x\in A\otimes H^0$. Then by Equation (2),
\begin{align*}
& ||(\lambda(u_{\gamma})-\lambda(u_{\gamma'}))(\alpha\otimes\id)(\rho(a)x)||
=||\lambda((u_{\gamma}-u_{\gamma'})\alpha(a))(\alpha\otimes\id)(x)|| \\
& \leq ||\lambda((u_{\gamma}-u_{\gamma'})\alpha(a))||\, ||x||
=||(u_{\gamma}-u_{\gamma'})\alpha(a)|| \, ||x||
\end{align*}
since $\lambda$ is isometric. Since $\rho(A)(A\otimes H^0 )$ is dense in
$A\otimes H^0$, $\{\lambda(u_{\gamma})y\}$ is a Cauchy net for any $y\in A\otimes H^0$.
Similarly by Equation (1), $\{y\lambda(u_{\gamma})\}$ is also a Cauchy net for any $y\in A\otimes H^0$.
Thus $\{\lambda(u_{\gamma})\}$ converges to some element $u\in M(A\otimes H^0 )$ strictly.
We note
\begin{align*}
\lim_{\gamma\to\infty}\rho(u_{\gamma}) & =\lim_{\gamma\to\infty}\underline{\rho}(u_{\gamma})
=\underline{\rho}(\lim_{\gamma\to\infty}u_{\gamma})=\underline{\rho}(1)=1 , \\
\lim_{\gamma\to\infty}\alpha^{-1}(u_{\gamma}) & =\lim_{\gamma\to\infty}\underline{\alpha}^{-1}(u_{\gamma})
=\underline{\alpha}^{-1}(\lim_{\gamma\to\infty}u_{\gamma})=\underline{\alpha}^{-1}(1)=1 ,
\end{align*}
where the limits are taken under the strict topologies in $M(A\otimes H^0 )$ and $M(A)$, respectively and $\underline{\alpha}^{-1}$ is an
automorphism of $M(A)$ extending $\alpha^{-1}$ to $M(A)$, which is strictly continuous on $M(A)$.
Hence by Equations (3), (4), we can see that $u$ is a unitary element in $M(A\otimes H^0 )$.
Let $\{v_{\beta}\}$ be another approximate unit of $A$ and let $v$ be the limit
of $\lambda(v_{\beta})$ under the strict topology in $M(A\otimes H^0 )$. Then by the above discussion,
we have that 
$$
||(\lambda(u_{\gamma})-\lambda(v_{\beta}))(\alpha\otimes\id)(\rho(a)x)||
\leq ||(u_{\gamma}-v_{\beta})\alpha(a)||\, ||x||
$$
for any $a\in A$ and $x\in A\otimes H^0 $. Since $\rho(A)(A\otimes H^0 )$ is dense in
$A\otimes H^0 $, $u=v$.
\end{proof}

\begin{lemma}\label{lem:relation}Let $u$ be as in the proof of Lemma \ref{lem:unitary}.
Then $u$ satisfies that
$\lambda(x)=\rho(x)u$ for any $x\in X_{\alpha}$, 
$\rho(\alpha(a))=u(\alpha\otimes\id)(a)u^* $ for any $a\in A$ and that
$(\underline{\rho}\otimes\id)(u)(u\otimes 1^0 )=(\id\otimes\Delta^0 )(u)$.
\end{lemma}
\begin{proof}Let $\{u_{\gamma}\}$ be an approximate unit of $A$.
By Equation (1), for any $x\in X_{\alpha}$, $\lambda(xu_{\gamma})=\rho(x)\lambda(u_{\gamma})$.
Thus $\lambda(x)=\rho(x)u$. Also, by Equation (2) for any $a\in A$,
$$
\lambda(u_{\gamma}\alpha(a))=\lambda(u_{\gamma})(\alpha\otimes\id)(\rho(a)) .
$$
Hence $\lambda(\alpha(a))=u(\alpha\otimes\id)(\rho(a))$.
Since $\lambda(\alpha(a))=\rho(\alpha(a))u$ for any $a\in A$ by the above
discussion, for any $a\in A$, 
$$
\rho(\alpha(a))u=u(\alpha\otimes\id)(\rho(a)) .
$$
for any $a\in A$. Since $u$ is a unitary element in $M(A\otimes H^0 )$,
$$
\rho(\alpha(a))=u(\alpha\otimes\id)(\rho(a))u^*
$$
for any $a\in A$. Furthermore, for any $a\in A$
\begin{align*}
(\lambda\otimes\id)(\lambda(u_{\gamma}a)) & =(\lambda\otimes\id)(\lambda(u_{\gamma})(\alpha\otimes\id)(\rho(\alpha^{-1}(a))) \\
& =(\rho\otimes\id)(\lambda(u_{\gamma}))((\lambda\otimes\id)\circ(\alpha\otimes\id)\circ\rho\circ\alpha^{-1})(a)
\end{align*}
by Equations (1), (2). Thus Equation (2)
\begin{align*}
(\lambda\otimes\id)(\lambda(a)) & =(\underline{\rho}\otimes\id)(u)((\lambda\otimes\id)\circ(\alpha\otimes\id)\circ\rho\circ\alpha^{-1})(a) \\
& =\lim_{\gamma\to\infty}(\underline{\rho}\otimes\id)(u)
(\lambda\otimes\id)((u_{\gamma}\otimes 1^0 )((\alpha\otimes\id)\circ\rho\circ\alpha^{-1})(a)) \\
& =\lim_{\gamma\to\infty}(\underline{\rho}\otimes\id)(u)(\lambda(u_{\gamma})\otimes 1^0 )((\alpha\otimes\id\otimes\id)\circ(\rho\otimes\id)\circ
\rho\circ\alpha^{-1})(a)) \\
& =(\underline{\rho}\otimes\id)(u)(u\otimes 1^0 )((\alpha\otimes\id\otimes\id)\circ(\id\otimes\Delta^0 )\circ\rho\circ\alpha^{-1})(a) \\
& =(\underline{\rho}\otimes\id)(u)(u\otimes\ 1^0 )((\id\otimes\Delta^0 )\circ(\alpha\otimes\id)\circ\rho\circ\alpha^{-1})(a) .
\end{align*}
Also, by Equation (2)
\begin{align*}
(\id\otimes\Delta^0 )(\lambda(u_{\gamma}a)) & =(\id\otimes\Delta^0 )(\lambda(u_{\gamma})((\alpha\otimes\id)\circ\rho\circ\alpha^{-1})(a)) \\
& =(\id\otimes\Delta^0 )(\lambda(u_{\gamma}))((\id\otimes\Delta^0 )\circ(\alpha\otimes\id)\circ\rho\circ\alpha^{-1})(a) .
\end{align*}
Thus
$$
(\id\otimes\Delta^0 )(\lambda(a))=(\id\otimes\Delta^0 )(u)((\id\otimes\Delta^0 )\circ(\alpha\otimes\id)\circ\rho\circ\alpha^{-1})(a) .
$$
By Equation (6)
$$
[(\rho\otimes\id)(u)(u\otimes 1^0 )-((\id\otimes\Delta^0 )(u)]((\id\otimes\Delta^0 )\circ(\alpha\otimes\id)\circ\rho\circ\alpha^{-1})(a)=0
$$
for any $a\in A$. Therefore,
$$
(\rho\otimes\id)(u)(u\otimes 1^0 )=(\id\otimes\Delta^0 )(u) .
$$
\end{proof}

\begin{remark}\label{rem:exterior}By Lemma \ref{lem:relation}, we can see that the coaction $(\alpha\otimes\id)\circ\rho\circ\alpha^{-1}$
of $H^0 $ on $A$ is exterior equivalent to $\rho$.
\end{remark}
    
Conversely, let $u$ be a unitary element in $M(A\otimes H^0 )$ satisfying that
$$
\rho=\Ad(u)\circ(\alpha\otimes\id)\circ\alpha^{-1} , \quad
(\underline{\rho}\otimes\id)(u)(u\otimes 1^0 )=(\id\otimes\Delta^0 )(u) .
$$
Let $\lambda_u$ be the linear map from $X_{\alpha}$ to $X_{\alpha}\otimes H^0 $ defined by
$$
\lambda_u (x)=\rho(x)u
$$
for any $x\in X_{\alpha}$. Then by routine computations, we can see that
$\lambda_u$ is a coaction of $H^0$ on $X_{\alpha}$ with respect to $(A, A, \rho, \rho)$.

\begin{prop}\label{prop:image}With the above notations, the following conditions are
equivalent:
\newline
$(1)$ $[X_{\alpha}]\in \Ima f_{\rho}$,
\newline
$(2)$ There is a unitary element $u\in M(A\otimes H^0 )$ such that
$$
\rho=\Ad(u)\circ(\alpha\otimes\id)\circ\rho\circ\alpha^{-1}, \quad
(\underline{\rho}\otimes\id)(u)(u\otimes 1^0 )=(\id\otimes\Delta^0 )(u) .
$$
\end{prop}
\begin{proof}This is immediate by Lemma \ref{lem:relation} and the above discussion.
\end{proof}

Let $u$ be a unitary element in $M(A\otimes H^0 )$ satisfying Condition (2) in Proposition \ref{prop:image}.
Let $\lambda_u$ be as above. We call $\lambda_u$ the coaction of $H^0$ on $X_{\alpha}$ with respect to
$(A, A, \rho, \rho)$
\it
induced by
\rm
$u$.
\par
Let $\alpha, \beta\in \Aut(A)$ satisfying that there are unitary elements $u, v\in M(A\otimes H^0 )$
such that
\begin{align*}
\rho & =\Ad(u)\circ(\alpha\otimes\id)\circ\rho\circ\alpha^{-1} , \quad
(\underline{\rho}\otimes\id)(u)(u\otimes 1^0 )=(\id\otimes\Delta^0 )(u) ,\\
\rho & =\Ad(v)\circ(\beta\otimes\id)\circ\rho\circ\beta^{-1} , \quad
(\underline{\rho}\otimes\id)(v)(v\otimes 1^0 )=(\id\otimes\Delta^0 )(v) .
\end{align*}

\begin{lemma}\label{lem:product1}With the above notations, we have the following:
$$
(\underline{\rho}\otimes\id)(u(\underline{\alpha}\otimes\id)(v))(u(\underline{\alpha}\otimes\id)(v)\otimes 1^0 )
=(\id\otimes\Delta^0 )(u(\underline{\alpha}\otimes\id)(v)) .
$$
\end{lemma}
\begin{proof}By routine computations, we can see that
$$
((\alpha\circ\beta)\otimes\id)\circ\rho\circ(\alpha\circ\beta)^{-1}
=\Ad((\underline{\alpha}\otimes\id)(v^* ))\circ\Ad(u^* )\circ\rho .
$$
Thus we obtain that
$$
\rho=\Ad(u(\underline{\alpha}\otimes\id)(v))\circ((\alpha\circ\beta)\otimes\id)\circ\rho\circ(\alpha\circ\beta)^{-1} .
$$
Since $\rho\circ\alpha=\Ad(u)\circ(\alpha\otimes\id)\circ\rho$,
$$
(\underline{\rho}\otimes\id)((\underline{\alpha}\otimes\id)(v))
=(u\otimes 1^0 )(\underline{\alpha}\otimes\id\otimes\id)((\underline{\rho}\otimes\id)(v))(u\otimes 1^0 )^* .
$$
Hence by routine computations, we can see that
$$
(\underline{\rho}\otimes\id)(u(\underline{\alpha}\otimes\id)(v))(u(\underline{\alpha}\otimes\id)(v)\otimes 1^0 )
=(\id\otimes\Delta^0 )(u(\underline{\alpha}\otimes\id)(v)) .
$$
\end{proof}

Let  $\alpha, \beta$ and $u, v$ be as above.
Let $\lambda_u$ and $\lambda_v$ be coactions of $H^0$ on $X_{\alpha}$ and $X_{\beta}$
with respect to $(A, A, \rho, \rho)$ induced by $u$ and $v$, respectively.
Let $u\sharp v=u(\underline{\alpha}\otimes\id)(v)\in M(A\otimes H^0 )$.
By Lemma \ref{lem:product1}, we can define the coaction $\lambda_{u\sharp v}$ of
$H^0$ on $X_{\alpha\circ\beta}$ with respect to $(A, A, \rho, \rho)$induced by
$u\sharp v$. By easy computations, we can see that $X_{\alpha}\otimes X_{\beta}$
is isomorphic to $X_{\alpha\circ\beta}$ by an $A-A$-equivalence bimodule
isomorphism $\pi$
$$
\pi: X_{\alpha}\otimes_{A}X_{\beta}\rightarrow X_{\alpha\circ\beta} : x\otimes y\mapsto x\alpha(y) .
$$
We identify $X_{\alpha}\otimes_A X_{\beta}$ with $X_{\alpha\circ\beta}$ by the
above $A-A$-equivalence bimodule isomorphism $\pi$.

\begin{lemma}\label{lem:product2}With the above notations,
for $[X_{\alpha}, \lambda_u ]$, $[X_{\beta}, \lambda_v ]\in \Pic_H^{\rho}(A)$,
$$
[X_{\alpha}, \lambda_u ][X_{\beta}, \lambda_v ]=[X_{\alpha\circ\beta}, \lambda_{u\sharp v}]\in\Pic_H^{\rho}(A) ,
$$
where $u\sharp v=u(\underline{\alpha}\otimes\id)(v)\in M(A\otimes H^0 )$.
\end{lemma}
\begin{proof}By the definition of the product in $\Pic_H^{\rho}(A)$,
$$
[X_{\alpha}, \lambda_u ][X_{\beta}, \lambda_v ]=[X_{\alpha}\otimes_A X_{\beta}, \lambda_u \otimes\lambda_v ] .
$$
Hence it suffices to show that
$$
\pi(h\cdot_{\lambda_u \otimes\lambda_v}x\otimes y)=h\cdot_{\lambda_{u\sharp v}}\pi(x\otimes y)
$$
for any $x\in X_{\alpha}, y\in X_{\beta}$ and $h\in H$. For any $x\in X_{\alpha}, y\in X_{\beta}$ and $h\in H$,
\begin{align*}
\pi(h\cdot_{\lambda_u \otimes\lambda_v }x\otimes y) & =\pi([h_{(1)}\cdot_{\lambda_u }x]\otimes[h_{(2)}\cdot_{\lambda_v}y]) \\
& =\pi([h_{(1)}\cdot_{\rho}x]\widehat{u}(h_{(2)})\otimes[h_{(3)}\cdot_{\rho}y]\widehat{v}(h_{(4)})) \\
& =[h_{(1)}\cdot_{\rho}x]\widehat{u}(h_{(2)})\alpha([h_{(3)}\cdot_{\rho}y]\widehat{v}(h_{(4)})) .
\end{align*}
Since $\rho\circ\alpha=\Ad(u)\circ(\alpha\otimes\id)\circ\rho$,
\begin{align*}
\pi(h\cdot_{\lambda_u \otimes\lambda_v}x\otimes y) & =[h_{(1)}\cdot_{\rho}x][h_{(2)}\cdot_{\rho}\alpha(y)]\widehat{u}(h_{(3)})
\alpha(\widehat{v}(h_{(4)})) \\
& =[h_{(1)}\cdot_{\rho}x\alpha(y)](u(\underline{\alpha}\otimes\id)(v))\widehat{}(h_{(2)}) \\
& =h\cdot_{\lambda_{u\sharp v}}x\alpha(y) .
\end{align*}
Therefore, we obtain the conclusions.
\end{proof}

\begin{cor}\label{cor:inverse}With the above notations, for any $[X_{\alpha}, \lambda_u ]\in\Pic_H^{\rho}(A)$,
$$
[X_{\alpha}, \lambda_u ]^{-1}=[X_{\alpha^{-1}}, \lambda_{(\underline{\alpha}^{-1}\otimes\id)(u^* )}]\in\Pic_H^{\rho}(A) .
$$
\end{cor}
\begin{proof}
This is immediate by Lemma \ref{lem:product2} and routine computations.
\end{proof}

For any $\alpha\in\Aut(A)$, let $\rU_{\alpha}^{\rho}(M(A\otimes H^0 ))$ be the set of
all unitary elements $u\in M(A\otimes H^0 )$ satisfying that
$$
\rho=\Ad(u)\circ(\alpha\otimes\id)\circ\rho\circ\alpha^{-1}, \quad
(\underline{\rho}\otimes\id)(u)(u\otimes 1^0 )=(\id\otimes\Delta^0 )(u) .
$$

\begin{lemma}\label{lem:relation2}With the above notations, for any $\alpha\in\Aut(A)$, we have
the following:
\newline
$(1)$ For any $u\in\rU_{\id}^{\rho}(M(A\otimes H^0 ))$ and $v\in\rU_{\alpha}^{\rho}(M(A\otimes H^0 ))$,
$uv\in\rU_{\alpha}^{\rho}(M(A\otimes H^0 ))$,
\newline
$(2)$ For any $u, v\in\rU_{\alpha}^{\rho}(M(A\otimes H^0 ))$, $uv^* \in\rU_{\id}^{\rho}(M(A\otimes H^0 ))$.
\end{lemma}
\begin{proof}(1) This is immediate by Lemma \ref{lem:product2}.
\newline
(2) By Corollary \ref{cor:inverse}, $(\alpha^{-1}\otimes\id)(v^* )\in \rU_{\alpha^{-1}}^{\rho}(M(A\otimes H^0 ))$.
Hence $uv^* \in\rU_{\id}^{\rho}(M(A\otimes H^0 ))$.
\end{proof}

\begin{lemma}\label{lem:kernel1}Let $u\in\rU_{\id}^{\rho}(M(A\otimes H^0 ))$. Then the following conditions are
equivalent:
\newline
$(1)$ $[{}_A A_A , \lambda_u ]=[{}_A A_A , \rho ]$ in $\Pic_H^{\rho}(A)$,
\newline
$(2)$ There is a unitary element $w\in M(A)\cap A'$ such that
$u=(w^* \otimes 1^0 )\underline{\rho}(w)$.
\end{lemma}
\begin{proof}We suppose Condition (1). Then there is an $A-A$-equivalence bimodule
automorphism $\pi$ of ${}_A A_A$ such that
$$
\rho(\pi(x))=(\pi\otimes\id)(\lambda_u (x))=(\pi\otimes\id)(\rho(x)u)
$$
for any $x\in {}_A A_A$. We note that $\pi\in {}_A \BB_A ({}_A A_A )$ and that
$$
{}_A \BB_A ({}_A A_A )\cong A' \cap \BB_A (A_A )\cong A' \cap M(A) .
$$
Hence there is a unitary element $w\in A' \cap M(A)$ such that
$\pi(x)=wx$ for any $x\in A$. Thus for any $x\in A$
$$
\rho(wx)=(w\otimes 1^0 )\rho(x)u .
$$
Therefore $u=(w^* \otimes 1^0 )\underline{\rho}(w)$.
Next we suppose Condition (2). Let $\pi$ be the $A-A$-equivalence bimodule
automorphism of ${}_A A_A $ defined by $\pi(x)=wx$ for any $x\in {}_A A_A$.
Then for any $x\in {}_A A_A$
\begin{align*}
\rho(\pi(x)) &=\rho(wx)=\rho(xw)=\rho(x)\underline{\rho}(w)=\rho(x)(w\otimes 1^0 )u=(w\otimes 1^0 )\rho(x)u \\
& =(\pi\otimes\id)(\lambda_u (x)) .
\end{align*}
Thus we obtain Condition (1).
\end{proof}

\begin{cor}\label{cor:kernel2}Let $\alpha\in\Aut(A)$ and $u, v\in\rU_{\alpha}^{\rho}(M(A\otimes H^0 ))$.
Then the following conditions are equivalent:
\newline
$(1)$ $[X_{\alpha}, \lambda_u ]=[X_{\alpha}, \lambda_v ]$ in $\Pic_H^{\rho}(A)$,
\newline
$(2)$ There is a unitary element $w\in M(A)\cap A'$ such that
$u=(w^* \otimes 1^0 )\underline{\rho}(w)v$.
\end{cor}
\begin{proof}We suppose Condition (1). By Lemma \ref{lem:product2} and Corollary \ref{cor:inverse},
we can see that 
$[{}_A A_A, \lambda_{uv^*}]=[{}_A A_A , \rho]$ in $\Pic_H^{\rho}(A)$. Thus by Lemma \ref{lem:kernel1}
there is a unitary element in $w\in M(A)\cap A'$ such that $uv^* =(w^* \otimes 1^0 )\underline{\rho}(w)$.
Hence we obtain Condition (2). Conversely we suppose Condition (2). Then there is a unitary element
$w\in M(A)\cap A'$ such that $uv^* =(w^* \otimes 1^0 )\underline{\rho}(w)$. Hence
$[{}_A A_A, \lambda_{uv^*}]=[{}_A A_A , \rho]$ in $\Pic_H^{\rho}(A)$.
Since $[X_{\alpha}, \lambda_u ][X_{\alpha}, \lambda_v ]^{-1}=[{}_A A_A , \lambda_{uv^* }]$ in
$\Pic_H^{\rho}(A)$ by Lemma \ref{lem:product2} and Corollary \ref{cor:inverse}, $[X_{\alpha}, \lambda_u ]=[X_{\alpha}, \lambda_v ]$ in $\Pic_H^{\rho}(A)$.
\end{proof}

We shall compute $\Ker f_{\rho}$, the kernel of $f_{\rho}$. Let $[X, \lambda]\in\Pic_H^{\rho}(A)$.
Then by Proposition \ref {prop:image}, we can see that $[X]=[{}_A A_A ]$ in $\Pic(A)$ if and only if
there is a unitary element $u\in\rU_{\id}^{\rho}(M(A\otimes H^0 ))$ such that $[X, \lambda]=[{}_A A_A , \lambda_u ]$
in $\Pic_H^{\rho}(A)$. Furthermore by Corollary \ref{cor:kernel2}, $[{}_A A_A , \lambda_u ]=[{}_A A_A , \lambda_v ]$
in $\Pic_H^{\rho}(A)$ if and only if there is a unitary element $w\in M(A)\cap A'$ such that
$u=(w^* \otimes 1^0 )\underline{\rho}(w)v$, where $u, v\in\rU_{\id}^{\rho}(M(A\otimes H^0 ))$.
We define an equivalence relation in $\rU_{\id}^{\rho}(M(A\otimes H^0 ))$ as follows:
Let $u, v\in\rU_{\id}^{\rho}(M(A\otimes H^0 ))$, written $u\sim v$ if there is a unitary element
$w\in M(A)\cap A'$ such that
$$
u=(w^* \otimes 1^0 )\underline{\rho}(w)v .
$$
Let $\rU_{\id}^{\rho}(M(A\otimes H^0 ))/\! \sim$ be the set of all equivalence classes in $\rU_{\id}^{\rho}(M(A\otimes H^0 ))$.
We denote by $[u]$ the equivalence class of $u\in\rU_{\id}^{\rho}(M(A\otimes H^0 ))$.
By Lemma \ref{lem:relation2}, $\rU_{\id}^{\rho}(M(A\otimes H^0 ))$ is a group.
Hence $\rU_{\id}^{\rho}(M(A\otimes H^0 ))/\!\sim$ is a group by easy computations.

\begin{prop}\label{prop:kernel3}With the above notations, $\Ker f_{\rho}\cong\rU_{\id}^{\rho}(M(A\otimes H^0 ))/ \! \sim$
as groups.
\end{prop}
\begin{proof}Let $\pi$ be a map from $\rU_{\id}^{\rho}(M(A\otimes H^0 ))/\!\sim$ to $\Ker f_{\rho}$
defined by
$$
\pi([u])=[{}_A A_A , \lambda_u ]
$$
for any $u\in\rU_{\id}^{\rho}(M(A\otimes H^0 ))$. By the above discussions,
we can see that $\pi$ is well-defined and bijective. For any $u, v\in\rU_{\id}^{\rho}(M(A\otimes H^0 ))$
$$
\pi([u])\pi([v])=[{}_A A_A , \lambda_u ][{}_A A_A , \lambda_v ]=[{}_A A_A , \lambda_{uv}]=\pi([uv])
$$
by Lemma \ref {lem:product2}. Therefore, we obtain the conclusion.
\end{proof}

We recall that there is a homomorphism $\Phi$ of $\Aut_H^{\rho^s}(A^s )$ to $\Pic_H^{\rho^s}(A^s )$ defined by
$$
\Phi(\alpha)=[X_{\alpha}, \lambda_{\alpha}]
$$
for any $\alpha\in\Aut_H^{\rho^s}(A^s )$, where $\lambda_{\alpha}$ is a coaction of $H^0$ on $X_{\alpha}$
induced by $\rho^s$ (See Section \ref{sec:Picard}). Then the following results hold:

\begin{lemma}\label{lem:image3}With the above notations, for any $\alpha\in\Aut _H^{\rho^s}(A^s )$
$$
(f_{\rho^s}\circ\Phi)(\alpha)=[X_{\alpha}]
$$
in $\Pic(A^s )$. Furthermore, if $\widehat{\rho}(1\rtimes_{\rho}e)\sim(1\rtimes _{\rho}e)\otimes 1$ in $(A\rtimes_{\rho}H)\otimes H$,
then
$$
\Ima f_{\rho^s}=\{ [X_{\alpha}]\in\Pic(A^s ) \, | \, \alpha\in\Aut_H^{\rho^s}(A^s ) \} .
$$
\end{lemma}
\begin{proof}This is immediate by easy computations.
\end{proof}

Let $G$ be a subgroup of $\Pic(A^s )$ defined by
$$
G=\{ [X_{\alpha} ]\in\Pic(A^s ) \, | \, \alpha\in\Aut_H^{\rho^s}(A^s ) \} .
$$

\begin{thm}\label{thm:exact3}Let $H$ be a finite dimensional $C^*$-Hopf algebra with its
dual $C^*$-algebra $H^0$. Let $A$ be a unital $C^*$-algebra and $\rho$ a coaction
of $H^0$ on $A$ with $\widehat{\rho}(1\rtimes_{\rho}e)\sim(1\rtimes_{\rho}e)\otimes 1$
in $(A\rtimes_{\rho}H)\otimes H$. Let $A^s =A\otimes\BK$ and $\rho^s$ the coaction of $H^0$ on $A^s$
induced by $\rho$. Let $\rU_{\id}^{\rho}(M(A^s\otimes H^0 ))$ be the group of all unitary elements
$u\in M(A^s \otimes H^0 )$ satisfying that
$$
\rho^s =\Ad(u)\circ\rho^s , \quad (\underline{\rho^s}\otimes\id)(u)(u\otimes 1^0 )=(\id\otimes\Delta^0 )(u) .
$$
Then we have the following exact sequence:
$$
1\longrightarrow\rU_{\id}^{\rho^s}(M(A^s \otimes H^0 ))/\!\sim\,\longrightarrow\Pic_H^{\rho^s}(A^s )\longrightarrow G\longrightarrow 1 ,
$$
where ``$\sim$" is the equivalence relation in $\rU_{\id}^{\rho^s}(M(A^s \otimes H^0 ))$ defined in this section.
\end{thm}
\begin{proof}This is immediate by Proposition \ref{prop:kernel3} and Lemma \ref{lem:image3}.
\end{proof}

Let $A$ be a UHF-algebra of type $N^{\infty}$, where $N=\dim H$. Let $\rho$ be the coaction of $H^0$
on $A$ defined in \cite [Section 7]{KT2:coaction}, which has the Rohlin property.
We note that $\widehat{\rho}(1\rtimes_{\rho}e)\sim(1\rtimes_{\rho}e)\otimes 1$ in $(A\rtimes_{\rho}H)\otimes H$
by \cite [Definition 5.1]{KT2:coaction}.

\begin{cor}\label{cor:exact4}With the above notations, we have the following exact sequence:
$$
1\longrightarrow\rU_{\id}^{\rho^s}(M(A^s \otimes H^0 ))\longrightarrow\Pic_H^{\rho^s}(A^s )\longrightarrow G\longrightarrow 1 .
$$
\end{cor}
\begin{proof}Since $A^s$ is simple, $M(A^s)\cap (A^s )'=\BC 1$ by
Pedersen \cite [Corollary 4.4.8]{Pedersen:auto}. Therefore by Theorem \ref{thm:exact3},
we obtain the conclusion.
\end{proof}

\section{Equivariant Picard groups and crossed products}\label{sec:duality}
Let $(\rho, u)$ be a twisted coaction of $H^0$ on a unital $C^*$-algebra $A$.
Let $f$ be a map from $\Pic_H^{\rho, u}(A)$ to $\Pic_{H^0}^{\widehat{\rho}}(A\rtimes_{\rho, u}H)$
defined by
$$
f([X, \lambda])=[X\rtimes_{\lambda}H, \widehat{\lambda}]
$$
for any $[X, \lambda]\in\Pic_H^{\rho, u}(A)$. In this section, we shall show that $f$ is an isomorphism
of $\Pic_H^{\rho, u}(A)$ onto $\Pic_{H^0}^{\widehat{\rho}}(A\rtimes_{\rho, u}H)$.
We can see that $f$ is well-defined in a straightforward way.
We show that $f$ is a homomorphism of $\Pic_H^{\rho, u}(A)$ to
$\Pic_{H^0}^{\widehat{\rho}}(A\rtimes_{\rho, u}H)$.
Let $A$, $B$ and $C$ be unital $C^*$-algebras and $(\rho, u)$, $(\sigma, v)$ and
$(\gamma, w)$ be twisted coactions of $H^0$ on $A$, $B$ and $C$,
respectively. Let $\lambda$ be a twisted coaction of $H^0$ on an $A-B$-equivalence
bimodule $X$ with respect to $(A, B, \rho, u, \sigma, v )$.
Also, let $\mu$ be a twisted coaction of $H^0$ on a $B-C$-equivalence bimodule
$Y$ with respect to $(B, C, \sigma, v, \gamma, w )$. Let $\Phi$ be a linear map from
$(X\otimes_{B} Y)\rtimes_{\lambda\otimes\mu}H$ to
$(X\rtimes_{\lambda}H)\otimes_{B\rtimes_{\sigma, v}H} (Y\rtimes_{\mu}H)$
defined by
$$
\Phi(x\otimes y\rtimes_{\lambda\otimes\mu}h)=(x\rtimes_{\lambda}1)\otimes(y\rtimes_{\mu}h )
$$
for any $x\in X$, $y\in Y$ and $h\in H$. By routine computations, $\Phi$ is well-defined.
We note that $(X\rtimes_{\lambda}H)\otimes_{B\rtimes_{\sigma, v}H}(Y\rtimes_{\mu}H)$
is consisting of finite sums of elements in the form $(x\rtimes_{\lambda}1)\otimes(y\rtimes_{\mu}h)$
by the definition of $(X\rtimes_{\lambda}H)\otimes_{B\rtimes_{\sigma, v}H}(Y\rtimes_{\mu}H)$,
where $x\in X$, $y\in Y$ and $h\in H$. Hence we can see that $\Phi$ is bijective and its inverse map
$\Phi^{-1}$ is :
$$
(X\rtimes_{\lambda}H)\otimes_{B\rtimes_{\sigma, u}H}(Y\rtimes_{\mu}H)
\to(X\otimes_B Y)\rtimes_{\lambda\otimes\mu}H:(x\rtimes_{\lambda}1)\otimes(y\rtimes_{\mu}h)
\mapsto x\otimes y\rtimes_{\lambda\otimes\mu}h .
$$
Furthermore, we have the following lemmas:

\begin{lemma}\label{lem:inner}With the above notations,
\begin{align*}
{}_{A\rtimes_{\rho, u}H} \la \Phi(x\otimes y\rtimes_{\lambda\otimes\mu}h), \,
\Phi(z\otimes r\rtimes_{\lambda\otimes\mu}l) \ra & ={}_{A\rtimes_{\rho, u}H} 
\la x\otimes y\rtimes_{\lambda\otimes\mu}h, \, z\otimes r \rtimes_{\lambda\otimes\mu} l \ra , \\
\la \Phi(x\otimes y\rtimes_{\lambda\otimes\mu}h), \,
\Phi(z\otimes r\rtimes_{\lambda\otimes\mu}l) \ra_{C\rtimes_{\gamma, w}H} & =
\la x\otimes y\rtimes_{\lambda\otimes\mu}h, \, z\otimes r \rtimes_{\lambda\otimes\mu}l \ra_{C\rtimes_{\gamma, w}H}
\end{align*}
for any $x, z\in X$, $y, r\in Y$ and $h, l\in H$.
\end{lemma}
\begin{proof}
We can prove this lemma by routine computations. Indeed,
\begin{align*}
& {}_{A\rtimes_{\rho, u}H} \la \Phi(x\otimes y\rtimes_{\lambda\otimes\mu}h), \, 
\Phi(z\otimes r\rtimes_{\lambda\otimes\mu}l)\ra \\
& ={}_{A\rtimes_{\rho, u}H} \la (x\rtimes_{\lambda}1)\otimes(y\rtimes_{\mu}h), \, 
(z\rtimes_{\lambda}1)\otimes (r\rtimes_{\mu}l) \ra \\
& ={}_{A\rtimes_{\rho, u}H} \la (x\rtimes_{\lambda}1) \, {}_{B\rtimes_{\sigma, y}H} \la y\rtimes_{\mu}h, \,
r\rtimes_{\mu}l \ra , \, z\rtimes_{\lambda}1 \ra \\
& ={}_{A\rtimes_{\rho, u}H} \la (x\rtimes_{\lambda}1) \, ({}_B \la y , \,
[S(h_{(2)}l_{(3)}^* )^* \cdot_{\mu}r]\widehat{w}(S(h_{(1)}l_{(2)}^* )^* , \, l_{(1)}) \ra \\
& \rtimes_{\sigma, v}h_{(3)}l_{(4)}^* ), \, z\rtimes_{\lambda}1 \ra \\
& ={}_{A\rtimes_{\rho, u}H} \la x {}_B \la y, \, [S(h_{(2)}l_{(3)}^* )^* \cdot_{\mu} r]
\widehat{w}(S(h_{(1)}l_{(2)}^* )^* , \, l_{(1)}) \ra \rtimes_{\lambda}h_{(3)}l_{(4)}^* , \, z\rtimes_{\lambda}1 \ra \\
& ={}_A \la x \, {}_B \la y, \, [S(h_{(2)}l_{(3)}^* )^* \cdot_{\mu}r ]
\widehat{w}(S(h_{(1)}l_{(2)}^* )^* , l_{(1)}) \ra , \, [S(h_{(3)}l_{(4)}^* )^* \cdot_{\lambda}z ] \ra
\rtimes_{\rho, u}h_{(4)}l_{(5)}^* .
\end{align*}
On the other hand,
\begin{align*}
& {}_{A\rtimes_{\rho, u}H} \la x\otimes y\rtimes_{\lambda\otimes\mu}h, \, z\otimes r \rtimes_{\lambda\otimes\mu}l \ra \\
& ={}_A \la x\otimes y, \, [S(h_{(2)}l_{(3)}^* )^* \cdot_{\lambda\otimes\mu}z\otimes r]
\widehat{w}(S(h_{(1)}l_{(2)}^* )^* , \, l_{(1)}) \ra \rtimes_{\rho, u}h_{(3)}l_{(4)}^* \\
& ={}_A \la x\otimes y, \, [S(h_{(3)}l_{(4)}^* )^* \cdot_{\lambda}z]
\otimes[S(h_{(2)}l_{(3)}^* )^* \cdot_{\mu}r]\widehat{w}(S(h_{(1)}l_{(2)}^* )^*, \, l_{(1)} )\ra
\rtimes_{\rho, u}h_{(4)}l_{(5)}^* \\
& ={}_A \la x \, {}_B \la y, \, [S(h_{(2)}l_{(3)}^* )^* \cdot_{\mu}r ]
\widehat{w}(S(h_{(1)}l_{(2)}^* )^* , l_{(1)}) \ra , \, [S(h_{(3)}l_{(4)}^* )^* \cdot_{\lambda}z ] \ra
\rtimes_{\rho, u}h_{(4)}l_{(5)}^* .
\end{align*}
Thus we obtain that
$$
{}_{A\rtimes_{\rho, u}H} \la \Phi(x\otimes y\rtimes_{\lambda\otimes\mu}h), \,
\Phi(z\otimes r\rtimes_{\lambda\otimes\mu}l) \ra ={}_{A\rtimes_{\rho, u}H} 
\la x\otimes y\rtimes_{\lambda\otimes\mu}h, \, z\otimes r \rtimes_{\lambda\otimes\mu} l \ra .
$$
Similarly we obtain that
$$
\la \Phi(x\otimes y\rtimes_{\lambda\otimes\mu}h), \,
\Phi(z\otimes r\rtimes_{\lambda\otimes\mu}l) \ra_{C\rtimes_{\gamma, w}H}=
\la x\otimes y\rtimes_{\lambda\otimes\mu}h, \, z\otimes r \rtimes_{\lambda\otimes\mu}l \ra_{C\rtimes_{\gamma, w}H} .
$$
\end{proof}

\begin{lemma}\label{lem:equivariant}With the above notations, $\Phi$ is an
$A\rtimes_{\rho, u}H-C\rtimes_{\sigma, w}H$-equivalence bimodule isomorphism of
$(X\otimes_B Y)\rtimes_{\lambda\otimes\mu}H$ onto
$(X\rtimes_{\lambda}H)\otimes_{B\rtimes_{\sigma, v}H}(Y\rtimes_{\mu}H)$
satisfying that
$$
\Phi(\phi\cdot_{\widehat{\lambda\otimes\mu}}(x\otimes y\rtimes_{\lambda\otimes\mu}h))
=\phi\cdot_{\widehat{\lambda}\otimes\widehat{\mu}}\Phi(x\otimes y\rtimes_{\lambda\otimes\mu}h)
$$
for any $x\in X$, $y\in Y$, $h\in H$ and $\phi\in H^0$.
\end{lemma}
\begin{proof}By Lemma \ref{lem:inner} and the remark after Jensen and Thomsen \cite [Definition 1.1.18]{JT:KK},
we see that $\Phi$ is an $A\rtimes_{\rho, u}H-C\rtimes_{\sigma, w}H$-equivalence bimodule isomorphism of
$(X\otimes_B Y)\rtimes_{\lambda\otimes\mu}H$ onto
$(X\rtimes_{\lambda}H)\otimes_{B\rtimes_{\sigma, v}H}(Y\rtimes_{\mu}H)$. 
Furthermore, for any $x\in X$, $y\in Y$, $h\in H$ and $\phi\in H^0$,
\begin{align*}
\Phi(\phi\cdot_{\widehat{\lambda\otimes\mu}}(x\otimes y\rtimes_{\lambda\otimes\mu}h)) & =
\Phi(x\otimes y\rtimes_{\lambda\otimes\mu}h_{(1)}\phi(h_{(2)})) \\
& =(x\rtimes_{\lambda}1)\otimes (y\rtimes_{\mu}h_{(1)})\phi(h_{(2)}) \\
& =[\phi_{(1)}\cdot_{\widehat{\lambda}}(x\rtimes_{\lambda}1)]
\otimes[\phi_{(2)}\cdot_{\widehat{\mu}}(y\rtimes_{\mu}h)] \\
& =\phi\cdot_{\widehat{\lambda}\otimes\widehat{\mu}}\Phi(x\otimes y\rtimes_{\lambda\otimes\mu}h) .
\end{align*}
Therefore, we obtain the conclusion.
\end{proof}

\begin{cor}\label{cor:homo}Let $f$ be a map from $\Pic_H^{\rho, u}(A)$ to
$\Pic_{H^0}^{\widehat{\rho}}(A\rtimes_{\rho, u}H)$ defined by $f([X,  \lambda])
=[X\rtimes_{\lambda}H, \widehat{\lambda}]$ for any $[X, \lambda]\in\Pic_H^{\rho, u}(A)$.
Then $f$ is a homomorphism of $\Pic_H^{\rho, u}(A)$ to
$\Pic_{H^0}^{\widehat{\rho}}(A\rtimes_{\rho, u}H)$.
\end{cor}
\begin{proof}This is immediate by Lemma \ref{lem:equivariant}.
\end{proof}

Next, we construct the inverse homomorphism of $f$ of
$\Pic_{H^0}^{\widehat{\rho}}(A\rtimes_{\rho}H)$ to $\Pic_H^{\rho, u}(A)$.
First, we note the following:
Let $(\alpha, v)$ and $(\beta, z)$ be twisted coactions of $H^0$ on
unital $C^*$-algebras $A$ and $B$, respectively. We suppose that there is an isomorphism
$\Phi$ of $B$ onto $A$ such that $(\Phi\otimes\id)\circ\beta=\alpha\circ\Phi$ and $v=(\Phi\otimes\id)(z)$.
Let $(X, \lambda)\in\Equi_H^{\alpha, v}(A)$. We construct an element $(X_{\Phi}, \lambda_{\Phi})$
in $\Equi_H^{\beta, z}(B)$ from $(X, \lambda)\in\Equi_H^{\alpha, v}(A)$ and $\Phi$ as follows: Let
$X_{\Phi}=X$ as vector spaces. For any $x, y\in X_{\Phi}$ and $b\in B$,
\begin{align*}
& b\cdot x=\Phi(b)x, \quad x\cdot b=x\Phi(b) \\
& {}_B \la x, y \ra =\Phi^{-1}({}_A \la x, y \ra ), \quad \la x, y \ra_B =\Phi^{-1} (\la x, y \ra_A ) .
\end{align*}
We regard $\lambda$ as a linear map from $X_{\Phi}$ to $X_{\Phi}\otimes H^0$.
We denote it by $\lambda_{\Phi}$. Then $(X_{\Phi}, \lambda_{\Phi})$ is an
element in $\Equi_H^{\beta, z}(B)$. By easy computations, the map
$$
\Pic_H^{\alpha, v}(A)\to\Pic_H^{\beta, z}(B): [X, \lambda]\mapsto [X_{\Phi}, \lambda_{\Phi}]
$$
is well-defined and it is an isomorphism of $\Pic_H^{\alpha, v}(A)$ onto
$\Pic_H^{\beta, z}(B)$. By Corollary \ref{cor:homo}, there is the homomorphism $\widehat{f}$ of
$\Pic_{H^0}^{\widehat{\rho}}(A\rtimes_{\rho, u}H)$ to
$\Pic_H^{\widehat{\widehat{\rho}}}(A\rtimes_{\rho, u}H\rtimes_{\widehat{\rho}}H^0 )$
defined by
$$
\widehat{f}([Y, \mu])=[Y\rtimes_{\mu}H^0 , \widehat{\mu}]
$$
for any $[Y, \mu]\in\Pic_{H^0}^{\widehat{\rho}}(A\rtimes_{\rho, u}H)$.
By Proposition \ref{prop:nonunital}, there are an isomorphism $\Psi_A$ of $A\otimes M_N (\BC)$
onto $A\rtimes_{\rho, u}H \rtimes_{\widehat{\rho}}H^0$ and a unitary
element $U\in (A\rtimes_{\rho, u}H\rtimes_{\widehat{\rho}}H^0)\otimes H^0$
such that
\begin{align*}
\Ad(U)\circ\widehat{\widehat{\rho}} & =(\Psi_A \otimes\id_{H^0})\circ(\rho\otimes\id_{M_N (\BC)})\circ\Psi_A^{-1} , \\
(\Psi_A\otimes\id_{H^0}\otimes\id_{H^0})(u\otimes I_N ) & =(U\otimes 1^0 )(\widehat{\widehat{\rho}}\otimes
\id_{H^0})(U)(\id\otimes\Delta^0 )(U^* ) .
\end{align*}
Let $\overline{\rho}=(\Psi_A^{-1}\otimes\id_{H^0})\circ\widehat{\widehat{\rho}}\circ\Psi_A$.
By the above discussions, there is the isomorphism $g_1$ of
$\Pic_H^{\widehat{\widehat{\rho}}}(A\rtimes_{\rho, u}H\rtimes_{\widehat{\rho}}H^0 )$
onto $\Pic_H^{\overline{\rho}}(A\otimes M_N (\BC))$ defined by
$$
g_1 ([X, \lambda])=[X_{\Psi_A}, \, \lambda_{\Psi_A}]
$$
for any $[X, \lambda]\in\Pic_H^{\widehat{\widehat{\rho}}}(A\rtimes_{\rho, u}H\rtimes_{\widehat{\rho}}H^0 )$.
Furthermore, the coaction $\overline{\rho}$ of $H^0$ on $A\otimes M_N (\BC)$ is
exterior equivalent to the twisted coaction $(\rho\otimes\id, u\otimes I_N )$.
Indeed,
$$
\rho\otimes\id_{M_N (\BC)}=(\Psi_A^{-1}\otimes\id_{H^0})\circ\Ad(U)\circ\widehat{\widehat{\rho}}\circ\Psi_A
=\Ad(U_1 )\circ\overline{\rho} .
$$
Let $U_1 =(\Psi_A^{-1}\otimes\id_{H^0})(U)$. Since $(\Psi_A^{-1}\otimes\id_{H^0}\otimes\id_{H^0})\circ (\id\otimes\Delta^0 )=(\id\otimes\Delta^0 )
\circ(\Psi_A^{-1}\otimes\id_{H^0})$,
$$
u\otimes I_N =(U_1 \otimes 1^0 )(\overline{\rho}\otimes\id)(U_1 )(\id\otimes\Delta^0 )(U_1^* ).
$$

We also note the following: We consider twisted coactions $(\alpha, v)$ and $(\beta, z)$ of
$H^0$ on a unital $C^*$-algebra $A$. We suppose that
$(\alpha, v)$ and $(\beta, z)$ are exterior equivalent. Then there is a unitary element
$w$ in $A\otimes H^0$ such that
$$
\beta=\Ad(w)\circ\alpha, \quad z=(w\otimes 1^0 )(\rho\otimes\id)(w)v(\id\otimes\Delta^0 )(w^* )
$$
By Lemmas \ref{lem:extmorita} and \ref{lem:isom} and their proofs, there is the isomorphism
$g_2$ of $\Pic_H^{\alpha, v}(A)$ onto $\Pic_H^{\beta, z}(A)$ defined by
$g_2 ([X, \lambda])=[X, \, \Ad(w)\circ\lambda ]$ for any $[X, \lambda]\in\Pic_H^{\alpha, v}(A)$,
where $\Ad(w)\circ\lambda$ means a linear map from $X$ to $X\otimes H^0$ defined
by $(\Ad(w)\circ\lambda)(x)=w\lambda(x)w^*$
for any $x\in X$, which is a coaction of $H^0$ on $X\otimes H^0$
with respect to $(A, A, \beta, z, \beta, z)$. Since $\overline{\rho}$ and $(\rho\otimes\id, \, u\otimes I_N )$ are
exterior equivalent, by the above discussions, there is the isomorphism $g_2$ of
$\Pic_H^{\overline{\rho}}(A\otimes M_N (\BC))$ onto
$\Pic_H^{\rho\otimes\id_{M_N (\BC)}, \, u\otimes I_N}(A\otimes
M_N (\BC))$ defined by
$$
g_2 ([X, \lambda])=[X, \, \Ad(U_1 )\circ\lambda]
$$
for any $[X, \lambda]\in\Pic_H^{\overline{\rho}}(A\otimes  M_N (\BC))$. By easy
computations, $(\rho, u)$ is strongly Morita equivalent to $(\rho\otimes\id_{M_N (\BC)}, \, u\otimes I_N )$.
Hence by Lemma \ref{lem:isom} and its proof, there is the isomorphism $g_3$ of
$\Pic_H^{\rho, u}(A)$ onto $\Pic_H^{\rho\otimes\id_{M_N (\BC)}, \, u\otimes I_N}(A\otimes M_N (\BC))$
defined by
$$
g_3 ([X, \lambda])=[X\otimes M_N (\BC), \, \lambda\otimes\id_{M_N (\BC)}]
$$
for any $[X, \lambda]\in\Pic_H^{\rho, u}(A)$. Let $g=g_3^{-1}\circ g_2 \circ g_1 \circ\widehat{f}$.
Then $g$ is a homomorphism of $\Pic_{H^0}^{\widehat{\rho}}(A\rtimes_{\rho, u}H)$ to
$\Pic_H^{\rho, u}(A)$.

\begin{prop}\label{prop:identity}With the above notations, $g\circ f=\id$ on $\Pic_H^{\rho, u}(A)$.
\end{prop}
\begin{proof}Let $[X, \lambda]\in\Pic_H^{\rho, u}(A)$. By the definitions of $f, \widehat{f}, g_1$
and $g_2$,
$$
(g_2 \circ g_1 \circ\widehat{f}\circ f)([X, \lambda])=[(X\rtimes_{\lambda}H\rtimes_{\widehat{\lambda}}H^0 )_{\Psi_A}, \,
\Ad(U_1 )\circ(\widehat{\widehat{\lambda}})_{\Psi_A}] .
$$
Let $\Psi_X$ be the linear map from $X\otimes M_N (\BC)$ to
$X\rtimes_{\lambda}H\rtimes_{\widehat{\lambda}}H^0$ defined in
Proposition \ref{prop:dual2} and we regard $\Psi_X$ as an $A\otimes M_N (\BC)
-A\otimes M_N (\BC)$-equivalence bimodule isomorphism of $X\otimes M_N (\BC)$
onto $(X\rtimes_{\lambda}H\rtimes_{\widehat{\lambda}}H^0 )_{\Psi_A}$.
Also, since $\Ad(U)\circ\widehat{\widehat{\lambda}}=(\Psi_X \otimes\id)\circ(\lambda\otimes\id)\circ\Psi_X^{-1}$
by Proposition \ref{prop:dual2}, for any $x\in A\otimes M_N (\BC)$,
\begin{align*}
(\Ad(U_1 )\circ(\widehat{\widehat{\lambda}})_{\Psi_A})(x) & =
U_1\cdot(\widehat{\widehat{\lambda}})_{\Psi_A}(x)\cdot U_1^*
=U\widehat{\widehat{\lambda}}(x)U^* \\
& =((\Psi_X \otimes\id)\circ(\lambda\otimes\id)\circ\Psi_X^{-1})(x) .
\end{align*}
Thus
$$
[(X\rtimes_{\lambda}H\rtimes_{\widehat{\lambda}}H^0 )_{\Psi_A}, \, 
\Ad(U_1 )\circ(\widehat{\widehat{\lambda}})_{\Psi_A}]=[X\otimes M_N (\BC), \, \lambda\otimes\id]
$$
in $\Pic_H^{\rho\otimes\id_{M_N (\BC)}, \, u\otimes I_N}(A\otimes M_N (\BC))$.
Since $g_3 ([X, \lambda])=[X\otimes M_N (\BC), \, \lambda\otimes\id_{M_N (\BC)}]$,
we obtain the conclusion.
\end{proof}

\begin{thm}\label{thm:duality}Let $(\rho, u)$ be a twisted coaction of $H^0$ on a unital
$C^*$-algebra $A$. Then $\Pic_H^{\rho, u}(A)\cong\Pic_{H^0}^{\widehat{\rho}}(A\rtimes_{\rho, u}H)$.
\end{thm}
\begin{proof}Let $f, \widehat{f}, g_i , \, (i=1,2,3)$ and $g$ as in the proof of
Proposition \ref{prop:identity}. By Proposition \ref{prop:identity}, $g\circ f=\id$ on
$\Pic_H^{\rho, u}(A)$. Hence $f$ is injective and $g$ is surjective. Furthermore,
we can see that $\widehat{f}$ is injective by Proposition \ref{prop:identity}. Since $g=g_3^{-1}\circ g_2 \circ g_1 \circ \widehat{f}$
and $g_i $ $(i=1,2,3)$ are bijective, $g$ is injective. It follows that $g$ is bijective.
Therefore, $f$ is an isomorphism of $\Pic_H^{\rho, u}(A)$ onto
$\Pic_{H^0}^{\widehat{\rho}}(A\rtimes_{\rho, u}H)$.
\end{proof}

\end{document}